\documentclass[11pt]{amsart}
\usepackage{amsmath,amscd,amssymb,amsfonts,amsthm,hyperref}
\usepackage{tcolorbox}
\usepackage{tikz-cd}
\hypersetup{colorlinks,   linkcolor={red!50!black},    citecolor={blue!70!black},   urlcolor={blue!80!black}}

\usepackage[cmtip,matrix,arrow]{xy}

\newtheorem{theorem}{Theorem}[section]

\newtheorem{proposition}[theorem]{Proposition}

\newtheorem{lemma}[theorem]{Lemma}
\theoremstyle{definition}    
\newtheorem{definition}[theorem]{Definition}
\theoremstyle{remark}

\newtheorem{remark}[theorem]{Remark}
\newtheorem{remarks}[theorem]{Remarks}
\newtheorem{example}[theorem]{Example}
\newcommand{\Ra}{\Rightarrow}
\newcommand\A{\mathcal{A}}
\newcommand{\EE}{\mathbb{E}}
\newcommand{\Cour}[1]      {[\![#1]\!]}

\newcommand\G{\mathcal{G}}

\renewcommand{\L}{\mathcal{L}}

\newcommand{\T}{\mathbb{T}}

\newcommand{\ca}{\mathcal}

\newcommand{\R}{\mathbb{R}}

\newcommand\pt{\on{pt}}

\newcommand{\ff}{f}
\newcommand{\hh}{h}

\newcommand\lie[1]{\mathfrak{#1}}
\renewcommand{\k}{\lie{k}}

\newcommand{\g}{\lie{g}}

\renewcommand{\a}{\mathsf{a}}

\newcommand{\on}{\operatorname}

\newcommand{\Ad}{ \on{Ad} }
\newcommand{\ad}{\on{ad}}

\newcommand{\Hom}{ \on{Hom}}

\renewcommand{\ker}{ \on{ker}}

\newcommand{\Mult}{  \on{Mult}}

\newcommand{\Pair}{\on{Pair}}

\newcommand{\id}{\on{id}}

\newcommand{\da}{\dasharrow}

\newcommand\qu{/\kern-.7ex/} 
\newcommand{\TG}{\mathbb{T}G}

\newcommand{\hra}{\hookrightarrow}

\renewcommand{\d}{{\mbox{d}}}
\newcommand{\ol}{\overline}

\newcommand{\f}{\frac}

\newcommand{\p}{\partial}
\renewcommand{\l}{\langle}
\renewcommand{\r}{\rangle}

\newcommand{\eeq}{\end{eqnarray*}}
\newcommand{\beq}{\begin{eqnarray*}}

\newcommand{\pr}{\on{pr}}
\newcommand{\wh}{\widehat}
\newcommand{\wt}{\widetilde}
\newcommand{\mf}{\mathfrak}
\newcommand{\rra}{\rightrightarrows}
\newcommand{\sz}{\mathsf{s}}
\newcommand{\tz}{\mathsf{t}}
\newcommand{\core}{{\on{core}}}

\newcommand{\LA}{\ca{LA}}
\newcommand{\VB}{\ca{VB}}
\newcommand{\CA}{\ca{CA}}

\newcommand{\gr}{\on{gr}}
\newcommand{\dd}{\mathfrak{d}}

\begin{document}
\sloppy
\title{On the integration of Manin pairs}
\author{David Li-Bland}
\author{Eckhard Meinrenken}

\begin{abstract}
It is a remarkable fact that the integrability of a Poisson manifold to a symplectic groupoid depends only on the 
integrability of its cotangent Lie algebroid $A$: The source-simply connected Lie groupoid $G\rra M$ 
integrating $A$ automatically acquires a multiplicative symplectic 2-form. More generally, a similar result holds for 
the integration of Lie bialgebroids to Poisson groupoids, as well as in the  `quasi' settings of Dirac structures and quasi-Lie bialgebroids. In this article, we will place these results into a general context of Manin pairs
$(\EE,A)$, thereby obtaining a simple geometric approach to these integration results. We also clarify the case where the 
groupoid $G$ integrating $A$ is not source-simply connected.   Furthermore, we obtain a description of Hamiltonian spaces for 
Poisson groupoids and quasi-symplectic groupoids within this formalism. 
\end{abstract}

\maketitle

\tableofcontents
\section{Introduction}
A central theme in Poisson geometry is the notion of a \emph{symplectic groupoid} 
\cite{cos:gro,kar:ana,wei:sym}, given by a Lie groupoid $G\rra M$ together with a 
multiplicative symplectic 2-form $\omega\in \Omega^2(\G)$. 
The units of a symplectic groupoid acquire a unique Poisson structure $\pi_M$, with the property that the target-source map 
\begin{equation}\label{eq:ts}
(\tz,\sz)\colon G\to M\times M\end{equation}
is a Poisson map, using the Poisson structure $(\pi_M,-\pi_M)$ on the pair groupoid $M\times M$. 

Generalizing symplectic groupoids, one has \emph{Poisson Lie groupoids} $(G,\pi_G)$ \cite{wei:coi}, with infinitesimal objects given by Lie bialgebroids $(A,A^*)$. 
Mackenzie--Xu \cite[Theorem 4.1]{mac:int} proved that the integrability of a Lie bialgebroid $(A,A^*)$ to a Poisson Lie groupoid depends only on the integrability of the Lie algebroid $A$.  If it is integrable, then 
its source-simply connected integration $G\rra M$ has a unique multiplicative Poisson structure integrating the Lie bialgebroid. As a special case, the integrability of a Poisson manifold depends only on the integrability of its cotangent Lie algebroid. 
See \cite[Section 14.3]{cra:lec} for a detailed discussion. 

These results generalize to `quasi' settings: Bursztyn, Crainic, Weinstein, and Zhu \cite{bur:int} constructed integrations of (exact) Dirac manifolds to quasi-symplectic groupoids \cite{xu:mom}, while   Iglesias-Ponte, Laurent-Gengoux, and Xu 
\cite{igl:uni} obtained the integration of \emph{quasi-Lie bialgebroids} \cite{roy:qua} to quasi-Poisson groupoids. The special case where $M$ is a point recovers the theory of quasi-Poisson Lie groups as integrations of quasi-Lie bialgebras, developed by Kosmann-Schwarzbach \cite{kos:qua1,kos:jac}. 

In this note, we put these results into the framework of morphisms of Manin pairs.  Recall that a \emph{Manin pair  $(\EE,A)$ over a manifold $M$} \cite{liu:ma}  is a Courant algebroid $\EE\to M$  together with an integrable Lagrangian subbundle $A\subseteq \EE$. The classical Dirac structures on manifolds \cite{cou:di,couwein:beyond} correspond to the  Courant algebroid $\T M=TM\oplus T^*M$. Suppose the Lie algebroid $A$ integrates to a Lie groupoid $G\rra M$  (not necessarily source-simply connected). The Courant algebroid $\T G=TG\oplus T^*G$ has the structure of a  $\CA$-groupoid \cite{lib:cou,meh:qgr}
\[ \T G\rra TM\oplus A^*.\]
Similarly, the pair groupoid is a $\CA$-groupoid, 
$\EE\times\ol{\EE}\rra \EE$.  
The Courant bracket defines a representation of the jet algebroid $\wh{A}=J^1(A)$ on $\EE$; we assume that it integrates  
to a groupoid representation $\wh{g}\mapsto \Ad_{\wh{g}}$ of the jet groupoid $\wh{G}=J^1(G)$, with certain natural 
properties. (See Definition \ref{def:equivariantmaninpair} below.)  Such an integration is automatic whenever $G$ is source-simply connected.
\medskip

\noindent{\bf Theorem A1.} 
\emph{
	Let 	$(\EE,A)$ be a $\wh{G}$-equivariant Manin pair over $M$. Then there is a distinguished  $\wh{G}\times \wh{G}$-equivariant 
	morphism of multiplicative Manin pairs
	\begin{equation}\label{eq:R} R\colon  (\TG,TG)\da (\EE\times\ol{\EE},A\times A),\end{equation}
	with base map $(\tz,\sz)\colon G\to M\times M$. 
}
\medskip

The theorem will be continued below, by providing an explicit formula for $R$. We refer to \eqref{eq:R} as an integration of the Manin pair $(\EE,A)$ since it generalizes the Poisson map \eqref{eq:ts}.
The existence of such a morphism (for the source-simply connected case)  was stated in the doctoral dissertation of the first author \cite[Section 6.2.1]{lib:th}. There, it was deduced from its  infinitesimal counterpart, given by a morphism of  infinitesimally multiplicative Manin pairs
\begin{equation}\label{eq:lib}
R_0\colon (\T A,TA)\da (\on{flip}(T\EE),\on{flip}(TA));\end{equation}
here $\on{flip}$ indicates an interchange of horizontal and vertical structures. 
In turn, $R_0$ was obtained using super-geometric techniques. 

The integration result for  Lie bialgebroids \cite{mac:int} corresponds to the situation that one is given a Dirac structure $B\subset \EE$ complementary to $A$. The pre-image of $B\times B$ under the morphism $R$ is then a multiplicative Dirac structure complementary to $TG\subset \TG$, or equivalently the graph of a multiplicative Poisson structure on $G$. For quasi-Lie algebroids \cite{igl:uni}, the 
complement $B$ need not be integrable, and one arrives at a quasi-Poisson structure. Similarly, if $\EE$ is an \emph{exact} Courant algebroid, the result on integration of twisted Dirac manifolds \cite{bur:int}  is recovered by choice of an isotropic splitting of the anchor map $\a_\EE\colon \EE\to TM$. Indeed, for exact Courant algebroids the morphism $R$ is described by a multiplicative 2-form. In particular, for a Poisson manifold $(M,\pi_M)$ with cotangent Lie algebroid $A=\on{gr}(\pi_M^\sharp)\subset \T M$, the multiplicative morphism of Manin pairs $R$  defines a multiplicative symplectic structure on $G$, in such a way that $(\tz,\sz)$ is a Poisson map. Note that we obtain a mild generalization of these integration results, since $G$ need not be source-simply connected. 

Notice that \eqref{eq:R} 
will restrict, in particular, to a comorphism of Lie algebroids $TG\da A\times A$ over $(\tz,\sz)$, which is described on  the level of sections by a Lie algebra 
morphism $\Gamma(A)\times \Gamma(A)\to \mf{X}(G)$. We will see that this is simply the map 
\[ (\tau,\sigma)\mapsto \sigma^L-\tau^R,\]
associating to section of $A\times A$ the corresponding left-invariant and right-invariant vector fields. Our formula for \eqref{eq:R} will extend this relation to the the ambient Courant algebroids. We take its restriction to the units to be the graph of the map 
\[ \alpha\colon (\EE\times \EE)^{(0)}=\EE\to \TG^{(0)}=TM\oplus A^*,\]
given by  $\alpha=(\a_\EE,-\pr_{A^*})$, where $\pr_{A^*}\colon \EE\to \EE/A\cong A^*$ is the projection. On the other hand, 
we have a maps between the cores (see Section \ref{subsec:basic})
\[ \beta\colon  \core(\TG)=A\oplus T^*M\to \core(\EE\times \EE)=\EE,\]
given by $\beta=(\iota_A,-\a_\EE^*)$. 
Using this notation, we 
have: \medskip

\noindent{\bf Theorem A2.} 
\emph{
	For $x\in \T G|_g$ and $(\zeta',\zeta)\in (\EE\times\ol{\EE})|_{(\tz(g),\sz(g))}$, we have 
	$x\sim_R (\zeta',\zeta)$ if and only if 
	\begin{equation}\label{eq:explicit0}
	\tz_{\TG}(x)=\alpha(\zeta'),\ \ \ \ \ \beta([l_{\wh{g}^{-1}}(x)])=\zeta-\Ad_{\wh{g}^{-1}}\zeta',\end{equation}
	for any choice of lift $\wh{g}\in \wh{G}$. 
	Here $l_{\wh{g}^{-1}}(x)\in \TG|_M$ denotes the left translate of $x$, and $[l_{\wh{g}^{-1}}(x)]$ 
	its image in $\core(\TG)$. 
}
\medskip

We shall see that \eqref{eq:explicit0}  is the unique multiplicative morphism of Manin pairs that is equivariant for the left $\wh{G}$-action and whose restriction to units is given by $\alpha$. 
One can also use right translation, leading to another version of \eqref{eq:explicit0}. 

\medskip
Similar to the description of quasi-Poisson manifolds in terms of morphisms of Manin pairs \cite{bur:cou}, 
one might take the morphism $R$ itself to be an intrinsic (splitting-free) definition of a quasi-Poisson groupoid (or quasi-symplectic groupoid, in the exact case). According to Theorem A, it is fully encoded in the $\wh{G}$-equivariant Manin pair $(\EE,A)$. 
We define a \emph{Hamiltonian $G$-space} for such a Manin pair  to be a $G$-manifold $P$, with action along a
\emph{moment map} $\Phi\colon P\to M$, together with a 
$\wh{G}$-equivariant  morphism of Manin pairs
\[ L\colon (\T P,TP)\da (\EE,A)\]
having $\Phi$ as its base map.  This is simply the $\wh{G}$-equivariant version of the definition in \cite{igl:ham,lib:cou}. One recovers  Xu's notion of 
Hamiltonian spaces for quasi-symplectic groupoids \cite{xu:mor}, as well as (quasi-)Poisson actions of (quasi-)Poisson groupoids, 
by using the following observation: \medskip

\noindent{\bf Theorem B.} 
\emph{Given a Hamiltonian $G$-space $P$ for a $\wh{G}$-equivariant Dirac structure $(\EE,A)$, the morphism $L$ 
	is multiplicative in the sense that 
	for $x\in \TG,\ y\in \T P, (\zeta',\zeta)\in \EE\times\ol{\EE}$, 
	\[ \Big(x\sim_R (\zeta',\zeta),\ y\sim_L\zeta,\ \sz_{\TG}(x)=\Phi_{\T P}(y)\Big)\ \ \Rightarrow\ \ x\cdot y\sim \zeta'.\]  
	Here the right hand side uses the $\TG\rra TM\oplus A^*$-action on $\T P$ along $\Phi_{\T P}\colon \T P\to TM\oplus A^*$. 
}
\medskip

\medskip

The paper is organized as follows. Section \ref{sec:vbg} provides background material on $\VB$-groupoids and $\LA$-groupoids. In particular, we discuss the notion of \emph{fat groupoid} for a $\VB$-groupoid. We will also describe a version of Theorem A  for $\VB$-groupoids, which is not difficult to prove but naturally leads to formulas similar to  
\eqref{eq:explicit0}. Section \ref{sec:manin} contains the main theorem on integration of $\wh{G}$-equivariant Manin pairs. 
Checking that $R$ does indeed define a Courant morphism involves some bracket calculations, which are deferred to Appendix \ref{app:A}.  
We also explain the infinitesimally multiplicative version \eqref{eq:lib} from \cite{lib:th}.
Section \ref{sec:special} describes various special cases and applications. In particular, we show that 
for Manin pairs $(M\times \dd,M\times \g)$ coming from actions of Manin pairs of Lie algebras $(\dd,\g)$ (as in \cite{lib:cou}), 
with an integration of the $\g$-action to an action of a Lie group $G$, the integrability is automatic -- even when $G$ is not simply connected. Towards the end of Section \ref{sec:special}, we explain how the integration results for Dirac manifolds 
\cite{bur:int} and (quasi-) Lie bialgebroids \cite{mac:lie,igl:uni} follow from Theorem A. Finally, Section \ref{sec:hamiltonian}  contains a proof of Theorem B and a discussion of Hamiltonian spaces for equivariant Manin pairs, comparing to 
Hamiltonian spaces for quasi-symplectic groupoids. Appendix \ref{app:B}
gives additional information on fat groupoids of $\VB$-groupoids, notably a detailed description of the action of the 
kernel of $\wh{G}\to G$. 

\bigskip
{\bf Acknowledgments.} We thank Daniel \'Alvarez,  
Henrique Bursztyn, Rui Fernandes, and Maria Amelia Salazar  for helpful discussions. E.M.'s research was supported by an NSERC Discovery grant.

\bigskip
{\bf Notation and conventions.}
Lie groupoids are indicated by a double arrow $G\rra M$, with $G$ the manifold of arrows and $M\subset G$ the submanifold of units. We shall also use the notation $M=G^{(0)}$ for the units. The source and target maps are denoted by $\sz,\tz\colon G\to M$, the groupoid multiplication of composable arrows by $g\circ h$, and groupoid inversion by $g\mapsto g^{-1}$. 
The pair groupoid of a manifold $M$ is denoted $\Pair(M)=M\times M\rra M$. Lie algebroids are denoted by an arrow $A\Ra M$; the anchor is denoted by $\a\colon A\to TM$. 
The Lie functor takes a Lie groupoid $G\rra M$ to a Lie algebroid
\[ A=\on{Lie}(G)\Ra M.\]
As a vector bundle, it is the normal bundle $A=\nu(G,M)$ of the units, with 
anchor induced by the difference $\sz_{TG}-\tz_{TG}\colon TG\to TM$ of the source and target maps for the tangent groupoid $TG\rra TM$. 

\section{$\VB$-groupoids}\label{sec:vbg}
We recall some background material on $\VB$-groupoids, referring to \cite{gra:vb,mac:gen} for more details. 

\subsection{Basic definitions}\label{subsec:basic}
A \emph{$\VB$-groupoid} is a Lie groupoid $J\rra B$, where the manifold of arrows $J$ has the structure of a vector bundle 
$J\to G$, in such a way that the graph of the groupoid multiplication is a sub-vector bundle $\gr(\Mult_J)\subset J\times J\times J$. (Equivalently, the scalar multiplication of $J$ is by groupoid morphisms.) This implies that base manifold is itself a Lie groupoid $G\rra M$, 
and that the units are a sub-vector bundle $B\to M$. 
The situation is depicted as  
\begin{equation}\label{eq:vbg}
{\begin{tikzcd}
	[column sep={8em,between origins},row sep={4.5em,between origins},]
	J\arrow[r,shift right] \arrow[r,shift left]\arrow[d]& B\arrow[d] \\
	G\arrow[r,shift right] \arrow[r,shift left]& M
	\end{tikzcd}}
\end{equation}
The \emph{core} of the $\VB$-groupoid is the quotient bundle 
\[ C=\core(J)=J|_M/B.\] 
Let $\sz_J,\tz_J\colon J\to B$ denote the source and target maps. Their difference vanishes on $B\subset J|_M$, and so descends to a bundle map 
\begin{equation}\label{eq:coreanchor}
\varrho\colon \core(J)\to B, \ \ [x]\mapsto \sz_J(x)-\tz_J(x)
\end{equation}
called the \emph{core-anchor} \cite{gra:vb}. We shall often identify $\core(J)\cong \ker(\tz_J)|_M$; in terms of this identification, $J|_M\cong C\oplus B$. 

The dual bundle of any $\VB$-groupoid $J\rra B$ with $\core(J)=C$ 
is a $\VB$-groupoid
$J^*\rra C^*$ with the unique groupoid  multiplication such that the pairing gives a groupoid morphism $J^*\oplus J\to \R$. 
Here $\core(J^*)=B^*$, and the core-anchor of $J^*$ is the dual $\varrho^*\colon B^*\to C^*$ of \eqref{eq:coreanchor}. 

An \emph{$\LA$-groupoid} 
%
is a $\VB$-groupoid with a vertical Lie algebroid structure $J\Ra G$,  
which is compatible with the groupoid structure in the sense that the graph of the groupoid multiplication 
$\gr(\Mult_J)\subset J\times J\times J$ 
is a Lie subalgebroid. The units of $J$ are a Lie subalgebroid $B\Ra M$, and  the core is  a Lie algebroid $C\Ra M$, with bracket on sections defined by the identification with left-invariant sections of $J$. (See below.) 
The core-anchor of an $\LA$-groupoid is a Lie algebroid morphism.

We shall mainly deal with the following examples of $\VB$-groupoids:
\begin{enumerate}

		\item\label{it:b}  The pair groupoid  $\on{Pair}(V)\rra V$ of a 
		 vector bundle $V\to M$  is a 
	$\VB$-groupoid; here $\core(\on{Pair}(V))=V$, with core-anchor the identity map. 
	The core projection $\Pair(V)|_M=V\oplus V\to V$ is given by $(v',v)\mapsto [(v',v)]=v-v'$. 
	If $A\Ra M$ is a Lie algebroid, then $\Pair(A)\rra A$ is an $\LA$-groupoid. 
	
		\item
	\label{it:a} The tangent bundle $TG\rra TM$ of a Lie groupoid $G\rra M$ is an $\LA$-groupoid, with 
	$\core(TG)=A$ the associated Lie algebroid $A\Ra M$, and core-anchor the usual Lie algebroid anchor $\a_A$. 
	The dual bundle is the $\VB$-groupoid $T^*G\rra A^*$, with $\core(T^*G)=T^*M$, and core-anchor $\a_A^*$.

\end{enumerate}

\subsection{Left and right actions}
The left multiplication of the groupoid $G\rra M$ on itself, $h\mapsto l_g(h)=g\circ h$, lifts to  
an action $x\mapsto l_g(x)=0_g\circ x$ on $\ker(\tz_J)$. This gives a bundle map $\ker(\tz_J)\to C=\core(J)$, with base map $\sz\colon G\to M$, taking $x\in \ker(\tz_J)|_g$ to $[l_{g^{-1}}x]\in C$. 
Similarly, using right translations we obtain a bundle map $\ker(\sz|_J)\to C$ with base map $\tz_J$. 
Both of these maps are fiberwise isomorphisms, resulting in bundle isomorphisms
\begin{equation}\label{eq:kert}
\ker(\tz_J)\cong \sz^*C,\ \ \ 
\ker(\sz_J)\cong \tz^*C.
\end{equation}
For sections $\sigma\in \Gamma(C)$, we denote by 
\[ \sigma^L\in \Gamma(\ker(\tz_J))\subset \Gamma(J),\ \ 
\sigma^R\in \Gamma(\ker(\sz_J))\subset \Gamma(J)\] 
the sections defined by these identifications.  Similarly, given $g\in G$ and elements  $\zeta\in C_{\sz(g)}$ and 
$\zeta'\in C_{\tz(g)}$, denote by  
\[ \zeta^L_g\in \ker(\tz_J)_g,\ \ \ \ (\zeta')^R_g\in\ker(\sz_J)_g\]
the elements such that $[l_{g^{-1}} (\zeta^L_g)]=\zeta,\ \ [r_{g^{-1}}((\zeta')^R_g)]=\zeta'$.  
Then 
\begin{equation}\label{eq:mult}
(\zeta^R_g)\circ 0_h=\zeta^R_{gh}, \ \ \  0_g\circ \zeta^L_h=
\zeta^L_{gh},\ \ \ \zeta^L_g\circ(-\zeta)^R_h=0_{gh}\end{equation}
under the groupoid multiplication for $J$. Furthermore, $\sz_J(\zeta_g^L)=\varrho(\zeta),\ \tz_J((\zeta')^R_g)=-\varrho(\zeta')$. 

If $J$ is an $\LA$-groupoid, one obtains a Lie algebroid structure on $C$ by the identification with left-invariant sections. This gives the  bracket relations
\begin{equation}\label{eq:bracket}
[\sigma_1^L,\sigma_2^L]=[\sigma_1,\sigma_2]^L,\ \ 
[\sigma_1^R,\sigma_2^R]=- [\sigma_1,\sigma_2]^R,\ \ 
[\sigma_1^L,\sigma_2^R]=0.
\end{equation}
For $J=TG$ this is the given Lie algebroid structure on 
$A=\on{Lie}(G)$.

\subsection{The fat groupoid}\label{subsec:fat}
Given a $\VB$-groupoid $J$ as above, there is an exact sequence of groupoids over $M$  
\[ 1\to K\to \wh{G}\to G\to 1.\]
Here the \emph{fat groupoid} \cite{gra:vb} $\wh{G}\rra M$  consists of elements $g\in G$ together 
with linear subspaces 
$\wh{g}\subset  J|_g$ that are complements to both $\ker(\tz_J)|_g$ and $\ker(\sz_J)|_g$. It is an open subset of the Grassmannian bundle of $\dim M$-dimensional subspaces, and hence is a manifold. 
The source and target of $\wh{g}$ 
are the source and target of  $g$, and for composable elements, 
\[ \wh{g}\circ \wh{h}=\{v\circ w|\ v\in \wh{g},\ w\in \wh{h}\mbox{ are composable in $J$}\}.\]
The subgroupoid $K$ consists of elements $\wh{g}\in \wh{G}$ such that $g=m$ is a unit; 
it is a family of Lie groups $K_m,\ m\in M$. We may regard $K$ as the fat groupoid of the restriction $J|_M\cong C\oplus B\rra B$. 
The groupoid $K\rra M$, is entirely described in terms of the core-anchor: 
\begin{equation}\label{eq:K} 
K_m=\{f\in \Hom(B_m,C_m)|\ \id_B+\varrho\circ \ff\mbox{ is invertible}\},\end{equation}
with product $\ff_1\circ \ff_2=\ff_1+\ff_2\circ (\id_B+\varrho\circ \ff_1)$. See Appendix \ref{app:B} for a detailed discussion.


%


\begin{example}\label{ex:fat}
	\begin{enumerate}
		\item The fat groupoid of $\Pair(V)\rra V$ is 
		the gauge groupoid of $V$. The subgroupoid $K$ is the general linear bundle, $K=\bigsqcup_{m\in M}\on{GL}(V_m)$.  
		\item 	For $TG\rra TM$, the fat groupoid is the jet groupoid $J^1(G)\rra M$. 
		\item The fat groupoid for $J$ is canonically isomorphic to that of $J^*$, by the map taking a subspace of $J|_g$ to its annihilator in $J^*|_g$.   
	\end{enumerate}
\end{example}

The fat groupoid comes with natural actions 
\begin{equation}\label{eq:manyactions}
\wh{G}\times \wh{G}\circlearrowright J,\ \ \ \ \wh{G}\circlearrowright J|_M,\ \ \ \ 
\wh{G}\circlearrowright J^{(0)},\ \ \ \ \wh{G}\circlearrowright \core(J).
\end{equation}
The first of these actions is given by $(\wh{g},\wh{h})\cdot x=l_{\wh{g}} (r_{\wh{h}^{-1}}(x))$. 
Here
\[ l_{\wh{g}}(x)=v\circ x,\]  
with the unique element $v\in \wh{g}\subset J_g$  that is composable with 
$x\in J$. Right translations are defined similarly. This $\wh{G}\times \wh{G}$-action restricts to a diagonal $\wh{G}$-action on $J|_M$, denoted by $\Ad$. For $x\in J|_m$ and $\sz(\wh{g})=m$ we have 
\[\Ad_{\wh{g}}\cdot x=l_{\wh{g}}r_{\wh{g}^{-1}}(x)\in J_{g\cdot m}.\]
The action preserves the units $J^{(0)}=B$,  
and so descends to an action, still denoted $\Ad$, on $C=\core(J)$. 
The various $\VB$-groupoid structure maps behave naturally for these actions:
\[ \tz_J(l_{\wh{g}}(x))=\Ad_{\wh{g}}\tz_J(x),\ \ \tz_J(r_{\wh{g}^{-1}}(x))=\tz_J(x),\]
\[ \sz_J(l_{\wh{g}}(x))=\sz_J(x),\ \  \sz_J(r_{\wh{g}^{-1}}(x))=\Ad_{\wh{g}}\sz_J(x),\]
\[ \varrho(\Ad_{\wh{g}}x)=\Ad_{\wh{g}}\varrho(x).\]
It is also clear that the left-action of $\wh{G}$ on $J$ extends the left-action of $G$ on $\ker(\tz_J)$, and similarly for the right action. Hence, 
\[ l_{\wh{g}}(\zeta^L_h)=\zeta^L_{gh},\ \ \ 
\ r_{\wh{g}^{-1}}(\zeta^L_{h'})=(\Ad_{\wh{g}}(\zeta))^L_{h' g^{-1}}
\]
and similarly for $(\zeta')^R_g$.

\subsection{A canonical comorphism}
We now describe a multiplicative comorphism for $\VB$-groupoids, which features many of the main ingredients of the integration problem for Manin pairs. (See Cattaneo-Dherin-Weinstein \cite{cat:int1} for background material on  comorphisms of vector bundles, Lie algebroids and Lie groupoids. In our discussion, the `comorphism' refers to the vertical structure.)  
By a \emph{multiplicative $\VB$-comorphism} $S\colon J_1\da  J_2$ between two $\VB$-groupoids, 
we mean a  vector bundle comorphism along some base map $\Phi\colon G_1\to G_2$ (i.e, given by a bundle map $\Phi^*J_2\to J_1$) whose graph $S\subset J_2\times J_1$ is a subgroupoid for the horizontal structure.  
From the definitions, one readily checks that this is equivalent to the dual map $S^*\colon J_1^*\to J_2^*$ being a $\VB$-groupoid morphism (i.e., a vector bundle map that is also a groupoid morphism).

\begin{proposition}\label{prop:vbgroupoids}
For every $\VB$-groupoid $J\rra B$, with $\core(J)=C$, there is a canonical $\wh{G}\times \wh{G}$-equivariant multiplicative $\VB$-comorphism 
\[{\begin{tikzcd}		[column sep={8em,between origins},row sep={4.5em,between origins},]	J\arrow[r,shift right] \arrow[r,shift left]\arrow[d]& B\arrow[d] \\	G\arrow[r,shift right] \arrow[r,shift left]& M	\end{tikzcd}}
\stackrel{S}{\dasharrow}
{\begin{tikzcd}		[column sep={8em,between origins},row sep={4.5em,between origins},]	\Pair(C)\arrow[r,shift right] \arrow[r,shift left]\arrow[d]& C\arrow[d] \\	\Pair(M)\arrow[r,shift right] \arrow[r,shift left]& M	\end{tikzcd}}
\] 
with base map $(\tz,\sz)\colon G\to \Pair(M)$. 
Explicitly,
\begin{equation}\label{eq:explicit-1}
 x\sim_S (\zeta',\zeta)\  \Leftrightarrow \ x=\zeta_g^L-(\zeta')^R_g\end{equation}
where $g$ is the base point of $x\in J$. If $J$ is an $\LA$-groupoid, then $S$ is a Lie algebroid comorphism over $(\tz,\sz)$. Accordingly, its dual  $S^*\colon J^*\to \Pair(B^*)$ is a morphism of Poisson 
$\VB$-groupoids. 
\end{proposition}
\begin{proof}
Define $S$ as the dual of the 	$\VB$-groupoid morphism $(\tz_{J^*},\sz_{J^*})\colon J^*\to \Pair(C^*)$. 
Since $S^*$  is $\wh{G}\times\wh{G}$-equivariant, the same is true of $S$.  
By definition of the morphism $S$ via duality, we have that 
$x\sim_S(\zeta',\zeta)$ 
if and only if for all $y\in J^*$, having the same base point $g\in G$ as $x$, 
\[ \l y,\,x\r=\l \sz_{J^*}(y),\zeta\r-\l \tz_{J^*}(y),\zeta'\r.
\]
But $\l \tz_{J^*}(y),\zeta'\r=\l y,\,(\zeta'_g)^R\r,\ \l \sz_{J^*}(y),\zeta\r=\l y,\,\zeta_g^L\r$. 
This gives \eqref{eq:explicit-1}. 
On the level of sections, $S$ is spanned by restrictions of sections 
\begin{equation}\label{eq:pairsections} \big((\sigma',\sigma),\sigma^L-(\sigma')^R\big)
\end{equation} 
of $\Pair(C)\times J$ over $\Pair(M)\times G$
with $\sigma,\sigma'\in \Gamma(J)$.
If $J$ is an $\LA$-groupoid, the bracket relations \eqref{eq:bracket} show that the Lie bracket of two sections of the form \eqref{eq:pairsections} is again of this form. Hence, $S$ is a sub-Lie algebroid. (Alternatively, in terms of $J^*$, we may use the fact that for every Poisson groupoid the target-source map is a Poisson map.) 
\end{proof}

\begin{example}
	Suppose $G\rra M$ is a Lie groupoid, with Lie algebroid $A$, and $TG\rra TM$ its tangent bundle.  
	The multiplicative $\LA$-comorphism  $TG\da \Pair(A)$ is dual to the morphism of Poisson groupoids 
	$T^*G\to \Pair(A^*)$ given by   
	the target-source map for the symplectic groupoid $T^*G\rra A^*$. \end{example}

Using the equivariance property, we can left translate $x\in J|_g$ by any choice of lift $\wh{g}\in \wh{G}$ 
to an element in $J|_M$, and 
decompose into its components in $J^{(0)}$ and $\ker(\tz_J)|_M\cong \core(J)$. We may thus 
re-write 
\eqref{eq:explicit-1} as 
\begin{equation}\label{eq:alt}
	x\sim_S (\zeta',\zeta)\ \  \Leftrightarrow\ \  \tz_J(x)=\varrho(\zeta'),\ \ \
[l_{\wh{g}^{-1}}(x)]=
\zeta-\Ad_{\wh{g}^{-1}}(\zeta'),\end{equation}
similar to the statement of Theorem A.


\section{Integration of Manin pairs}\label{sec:manin}
Our conventions for Courant algebroids and Dirac structures follow 
\cite{bur:cou,lib:cou,lib:dir}. In particular, we work with the non-skew symmetric version of the Courant bracket
$\Cour{\cdot,\cdot}$ on $\Gamma(\EE)$, also  known as the Dorfman bracket. The \emph{standard Courant algebroid} over $M$ is denoted $\T M=TM\oplus T^*M$, 
with metric given by the pairing of vectors and covectors, with anchor the projection to the first summand, 
and with bracket 
\begin{equation}\label{eq:CA}
 \Cour{X_1+\mu_1,X_2+\mu_2} =[X_1,X_2]+\L_{X_1}\mu_2-\iota_{X_2}\d\mu_1\end{equation}
for vector fields $X_1,X_2$ and one-forms $\mu_1,\mu_2$.

\subsection{Manin pairs}%
In the work of Courant and Weinstein \cite{cou:di,couwein:beyond}, Dirac structures on manifolds $M$ were defined as (wide) Lagrangian subbundles of $\T M$ whose space of sections is closed under the Courant bracket. 
Liu-Weinstein-Xu \cite{liu:ma} generalized $\T M$  to the notion of \emph{Courant algebroid}; 
the pair $(\EE,A)$ of a Courant algebroid $\EE\to M$ together with a Dirac structure $A\subset \EE$ 
is called a \emph{Manin pair} over $M$.
For $M=\pt$, one recovers the classical notion of a Manin pair of Lie algebras, as introduced by Drinfeld.

\begin{definition} \label{def:mor}\cite{al:der,bur:cou,sev:let}
A \emph{Courant morphism}   $R\colon \EE_1\da \EE_2$, with base map $\Phi\colon M_1\to M_2$, is a 
Lagrangian subbundle $R\subset \EE_2\times \ol{\EE_1}$ along $\on{gr}(\Phi)\subset M_2\times M_1$, with image under the anchor tangent to 
$\on{gr}(\Phi)$, and such that the space of sections of $\EE_2\times \ol{\EE_1}$ that restrict to sections of $R$ is closed under the Courant bracket. 
A (strong) \emph{Dirac morphism} \cite{al:pur,me:lec} or \emph{morphism of Manin pairs} \cite{bur:cou}
\[ R\colon (\EE_1,A_1)\da (\EE_2,A_2),\]
with base map $\Phi\colon M_1\to M_2$, is a Courant morphism $R\colon \EE_1\da \EE_2$ satisfying 
\begin{equation}\label{eq:condition}  R(A_1)\subseteq A_2,\ \ R^{-1}(0)\cap A_1=0.\end{equation}
\end{definition}

Here $R(A_1)$ is the image of $A_1$ under the relation, and $R^{-1}(0)=\ker(R)$ the pre-image of the zero subbundle $0\subset \EE_2$.  The condition \eqref{eq:condition} means that for all $m\in M_1$, every element $y_2\in A_2|_{\Phi(m)}$ is $R$-related to a \emph{unique} element  $y_1\in A_1|_{m}$. A  morphism of Manin pairs gives, in particular, a comorphism of Lie algebroids 
$ A_1\da A_2$
with base map $\Phi\colon M_1\to M_2$, and hence a Lie algebra morphism $\Gamma(A_2)\to \Gamma(A_1)$. Unlike Courant morphisms, morphisms of Manin pairs can always be composed.  Another basic fact (see e.g., \cite{al:pur})
is that  whenever $B_2\subset \EE_2$ is a 
Lagrangian subbundle (resp., Dirac structure) complementary to $A_2$, then its pre-image $B_1=R^{-1}(B_2)$ is a
Lagrangian subbundle (resp., Dirac structure) complementary to $A_1$. 

\begin{example}
A map $\Phi\colon M_1\to M_2$ between Poisson manifolds $(M_i,\pi_i)$ is Poisson if and only if 
it   defines a morphism of Manin pairs 
\[ \T\Phi\colon (\T M_1,\on{gr}(\pi_1^\sharp))\da (\T M_2,\on{gr}(\pi_2^\sharp)),\] 
where 
the relation $\T\Phi$ is defined as the sum of the (graphs of)  tangent and cotangent maps. 
\end{example}

\subsection{Multiplicative Manin pairs}
A \emph{$\CA$-groupoid} \cite{lib:dir,meh:qgr} $\EE\rra \EE^{(0)}$ 
is a $\VB$-groupoid with a vertical Courant algebroid structure over its base $G\rra G^{(0)}$,
in 
such a way that the  graph of the groupoid multiplication $\on{gr}(\Mult_\EE)\subset \EE\times \ol{\EE}\times \ol{\EE}$ 
is a Dirac structure with support on $\on{gr}(\Mult_G)\subset G\times G\times G$. Examples include the pair groupoid of a 
given Courant algebroid, as well as the standard Courant algebroid over a Lie groupoid. 
Some basic facts about $\CA$-groupoids (see  \cite[Proposition 2.3]{lib:dir}):
\begin{itemize}
	\item[(i)] The space of units $\EE^{(0)}$ is a Dirac structure with support on $G^{(0)}$. In particular, the 
	 metric induces a nondegenerate pairing between $\EE^{(0)}$ and $\EE|_M/\EE^{(0)}=\core(\EE)$. 
	\item[(ii)] The subbundles $\ker(\tz_\EE),\ \ker(\sz_\EE)$ are orthogonal for the metric on $\EE$. 
\end{itemize} 
A \emph{morphism of $\CA$-groupoids}
\begin{equation}\label{eq:morCA}
{\begin{tikzcd}		[column sep={8em,between origins},row sep={4.5em,between origins},]	\EE_1\arrow[r,shift right] \arrow[r,shift left]\arrow[d]& \EE_1^{(0)}\arrow[d] \\	G_1\arrow[r,shift right] \arrow[r,shift left]& M_1	\end{tikzcd}}
\stackrel{R}{\dasharrow}
{\begin{tikzcd}		[column sep={8em,between origins},row sep={4.5em,between origins},]	\EE_2\arrow[r,shift right] \arrow[r,shift left]\arrow[d]& \EE_2^{(0)}\arrow[d] \\	G_2\arrow[r,shift right] \arrow[r,shift left]& M_2	\end{tikzcd}}
\end{equation}
is a 
$\CA$-morphism $\EE_1\da \EE_2$ (given by a Dirac structure $R\subset \EE_2\times \ol{\EE}_1$ 
with support on the graph of the base map $\Phi\colon G_1\to G_2$) such that 
$R\subset \EE_2\times\ol{\EE}_1$ is a subgroupoid. In particular, the base map  $\Phi$ is a 
groupoid morphism.

\begin{lemma}\label{lem:think}
Let \eqref{eq:morCA} be a morphism of $\CA$-groupoids,  and suppose that the induced relation of units is a $\VB$-comorphism $\alpha\colon \EE_1^{(0)}\da \EE_2^{(0)}$. 
Then the induced relation of cores is the $\VB$-morphism 
$\alpha^*\colon \core(\EE_1)\to \core(\EE_2)$ dual to $\alpha$. 
\end{lemma}
\begin{proof}
Since $R$ is a Lagrangian relation, we have  that 
\[ x_1\sim_R x_2,\  y_1\sim_R y_2 \Rightarrow 
\l x_1,y_1\r=\l x_2,y_2\r.\] 
(Here we assume that $x_1,y_1$ have the same base point, and likewise for $x_2,y_2$.) 
Apply this to 	$x_i\in \EE_i^{(0)}$ and $y_i\in \core(\EE_i)\cong \ker(\tz_{\EE_i})|_{M_i}$. Since $x_1=\alpha(x_2)$
(thinking of the comorphism $\alpha$ as a bundle map 
$\Phi^*\EE_2^{(0)}\to \EE_1^{(0)}$), this shows $y_2=\alpha^*(y_1)$. 
\end{proof}

A \emph{multiplicative Manin pair} $(\EE,A)$ \cite{alv:poi,lib:th,lib:qua,ort:mu} is a Manin pair 
such that $\EE$ is a $\CA$-groupoid, with $A$ as a subgroupoid.
Basic examples include $(\TG,TG)$ and $(\TG,T^*G)$ for a Lie groupoid $G\rra M$, 
another example is the Cartan-Dirac structure $(G\times (\g\oplus \ol{\g}),G\times \g_\Delta)$ 
on a metrized Lie group \cite[Section 3.4]{al:pur}. If $(\EE,A)$ is a given Manin pair, then 
 the pair groupoid $(\Pair(\EE),\Pair(A))$ is a multiplicative Manin pair; here  $\Pair(\EE)=\EE\times \ol{\EE}$ carries the product Courant algebroid structure, after first reversing the metric on the second factor.

A \emph{morphism of multiplicative Manin pairs} 
$R\colon (\EE_1,A_1)\da (\EE_2,A_2)$
is a morphism of Manin pairs for which the Courant morphism $R$ is a subgroupoid of $\EE_2\times \ol{\EE}_1$. 
In particular, $R$ is a $\CA$-groupoid morphism $R\colon \EE_1\da \EE_2$.

\subsection{Equivariant Manin pairs}
If $A\Ra M$ is a Lie algebroid with an integration to a Lie groupoid $G\rra M$, there is a canonical action of the jet groupoid  
$J^1(G)\rra M$ on $A$, coming from the interpretation of $J^1(G)$ as the fat groupoid $\wh{G}$ of $TG\rra TM$ and its 
natural action on $\core(TG)\cong A$. (See \eqref{eq:manyactions}.) Given a Manin pair $(\EE,A)$, we are interested in an extension of this $J^1(G)$-action to the Courant algebroid $\EE$. 
This extension should satisfy certain natural properties. 

First, the $J^1(G)$-representation on $\EE$ should differentiate to the representation of the jet algebroid $\wh{A}=J^1(A)$ on $\EE$, given on the level of sections by by the Courant bracket: 
\[ \nabla_{j^1(\xi)}\zeta=\Cour{\xi,\zeta},\ \ \ \xi\in \Gamma(A),\ \zeta\in\Gamma(\EE).\]
As explained in Appendix \ref{app:B}, the action of $\k\subset J^1(A)$ integrates \emph{canonically} to an action of 
$K\rra M$ -- in fact, the groupoid $K$ and its action on $\EE$ do not depend on integrability of $A$. Finally, the $J^1(G)$-action on $\EE$ should preserve the Courant algebroid structure, in the sense that the anchor and metric are preserved, and the bracket is preserved under the action of local bisections of $G$ (i.e., the holonomic local bisections of $J^1(G)$).

\begin{definition}\label{def:equivariantmaninpair}
A \emph{$\wh{G}$-equivariant Manin pair} is a Manin pair $(\EE,A)$, together with a Lie groupoid $G\rra M$ integrating $A$ and an extension of the $\wh{G}=J^1(G)$-action on $A$ to  the Courant algebroid $\EE$, such that:
 	\begin{itemize}
 	\item The $\wh{G}$-action differentiates to the given $\wh{A}$-action.  	
 	\item The $\wh{G}$-action  restricts to the given $K$-action.  
 	\item The metric on $\EE$ is $\wh{G}$-invariant, the anchor $\a_\EE$ is $\wh{G}$-equivariant, and the bracket on $\Gamma(\EE)$ is preserved under the action of  (local) bisections of $G$.
 \end{itemize}	 
\end{definition}
If $G$ is source-simply connected, these conditions are automatic:
\begin{proposition}If $(\EE,A)$ is a Manin pair, and $G\rra M$  a source-simply connected integration of $A$, then 
there is a unique  $\wh{G}$-action on $\EE$ making it into a $\wh{G}$-equivariant Manin pair.
\end{proposition}
This result is not obvious, since the jet prolongation of a source-simply connected Lie groupoid is not even source connected, in general. (We thank R. Fernandes for alerting us to this fact, which may be illustrated by 
$G=\Pair(M)$ for a simply connected manifold $M$.)  We will give a proof in Appendix \ref{app:B} (see Proposition \ref{prop:rr}).
If $G$ is only source connected, then its lift to a $\wh{G}$-action as in  Definition \ref{def:equivariantmaninpair} may 
not exist, but is unique if it does. 
 

\begin{example}
If $M=\pt$, then $G$ is a Lie group, and coincides with its jet groupoid. A $\wh{G}$-equivariant Manin pair 
$(\EE,A)$ is a  Manin pair of Lie algebras $(\dd,\g)$, with $G$ acting by 
Lie algebra automorphisms of $\dd$, preserving the metric, extending the adjoint action on $\g$, and 
differentiating to the adjoint action of $\g$ on $\dd$. (See \cite{lib:dir,me:poilec}.)  

\end{example}

\begin{example}
To illustrate the significance of $\wh{G}$-equivariance, consider the Manin pair $(\T M,T^*M)$. The Lie algebroid structure of $A=T^*M$ has zero bracket; hence its source-simply connected integration is the vector bundle $T^*M$ itself,  with  groupoid structure given by fiberwise addition. More general source-connected integrations are of the form 
\[ G=T^*M/\Lambda,\]
where $\Lambda$ is a subgroupoid with discrete source fibers. We claim that the action of $\wh{A}=J^1(A)$ on $\T M$ integrates to an action of $\wh{G}=J^1(G)$ if and only if $\Lambda$ is a \emph{Lagrangian} submanifold
of $T^*M$.
Note that this is also the condition for the symplectic form on $T^*M$ 
to descend to a symplectic structure $\omega$ on the quotient $G=T^*M/\Lambda$. 
To prove the claim, consider the infinitesimal  action of $\nu\in \Gamma(A)=\Gamma(T^*M)$ on $\T M$, given by Courant bracket
\[ \Cour{\nu,X+\mu}=-\iota_X\d\nu.\]
It integrates to the action of the additive \emph{group} $\Gamma(T^*M)$ by 
\begin{equation}\label{eq:computed} \nu\cdot (X,\mu)\mapsto (X,\mu-\iota_X\d\nu).\end{equation}
This $\Gamma(T^*M)$-action descends to a $\Gamma(G)$-action (which, in turn, determines a $\wh{G}$-action) if and only if 
the local sections of $\Lambda$ act trivially. 
The formula shows that the action of a local section $\nu$ is 
 trivial if and only if $\nu$ (regarded as a 1-form) is \emph{closed}. It follows that the $\Lambda$ is locally generated by closed 1-forms;  equivalently, it is a Lagrangian submanifold. 
\end{example}

\subsection{Integration of Manin pairs}
Let $(\EE,A)$ be a $\wh{G}$-equivariant Manin pair. Then $(\Pair(\EE),\Pair(A))$ is a 
 $\wh{G}\times \wh{G}$-equivariant multiplicative Manin pair. On the other hand, 
 $\wh{G}\times \wh{G}$ also acts on the tangent and cotangent bundles of $G$, 
 and so acts on $\T G$. It induces adjoint actions of $\wh{G}$ on $\TG|_M,\ \TG^{(0)},$ and $\core(\TG)$, see \eqref{eq:manyactions}.
 
 Define 
 maps 
 \begin{equation}\label{eq:alphabeta}
 \alpha\colon  \Pair(\EE)^{(0)}\to(\TG)^{(0)},\ \ \ \beta\colon \core(\TG) \to\core(\Pair(\EE))
 \end{equation}
 as follows:
\begin{align*}
\alpha&=(\a_\EE,-\pr_{A^*})\colon \EE\to TM\oplus A^*,\\
\beta&=(\iota_A,-\a_\EE^*)\colon A\oplus T^*M\to \EE.
\end{align*}
(The signs can be made to disappear by working with 
$(\ol{\EE},A)$ in place of $(\EE,A)$.) The two maps are $\wh{G}$-equivariant 
(since $\iota_A,\ \a_\EE$ and their dual maps are), 
and are related to the core-anchors as follows:
\begin{lemma}\label{lem:coreanchor}
The following diagram commutes: 
\[{\begin{tikzcd}		[column sep={8em,between origins},row sep={4.5em,between origins},]	\core(\T G)\arrow[r,"\varrho"]\arrow[d,"\beta"]& \TG^{(0)} \\	\core(\Pair(\EE))\arrow[r,"\varrho"]& \Pair(\EE)^{(0)}\arrow[u,"\alpha"]	\end{tikzcd}}
\] 
where the horizontal maps are the core-anchors. 
\end{lemma}
\begin{proof}
	Commutativity of the diagram follows 
	from $\a_\EE\circ \iota_A=\a_A$ (the core-anchor for $TG$) and the dual equality $\pr_{A^*}\circ \a_\EE^*=\a_A^*$ (the core-anchor for $T^*G$). Note that the core-anchor for $\Pair(\EE)$ is the identity map, after identifying both core and units with $\EE$. 
\end{proof}

\begin{remark}\label{rem:confusing}
The maps $\alpha,\beta$ satisfy $\l x,\alpha(\zeta)\r=-\l \beta(x),\zeta\r$ for all $x\in A\oplus T^*M$ and 
$\zeta\in \EE$. Thought of as maps between units and cores respectively, $\beta$ is dual to $\alpha$ 
as in Lemma \ref{lem:think}. (The sign disappears since 
the pairing between $\Pair(\EE)^{(0)}\cong \EE$ and $\core(\Pair(\EE))\cong \EE$ is given by \emph{minus} the metric.)  
\end{remark}

Let $x\in \TG$ with base point $g\in G$. For any lift $\wh{g}\in \wh{G}$, the left translate
 $ l_{\wh{g}^{-1}}(x)$ lies in $\TG|_M$, with base point 
 $\sz(g)\in M\subset G$. 
The equivalence class of this element in 
\[ \core(\T G)=\T G|_M/(\T G)^{(0)}\]
will be denoted 
$[l_{\wh{g}^{-1}}(x)]$. 


\begin{theorem}[Integration of Manin pairs] \label{th:main}
Given a $\wh{G}$-equivariant Manin pair $(\EE,A)$, 	the map $\alpha\colon \EE\to TM\oplus A^*$ 
extends uniquely to a $\wh{G}\times \wh{G}$-equivariant $\CA$-groupoid morphism 
\[{\begin{tikzcd}		[column sep={8em,between origins},row sep={4.5em,between origins},]	\TG\arrow[r,shift right] \arrow[r,shift left]\arrow[d]& TM\oplus A^*\arrow[d] \\	G\arrow[r,shift right] \arrow[r,shift left]& M	\end{tikzcd}}
\stackrel{R}{\dasharrow}
{\begin{tikzcd}		[column sep={8em,between origins},row sep={4.5em,between origins},]	\Pair(\EE)\arrow[r,shift right] \arrow[r,shift left]\arrow[d]& \EE\arrow[d] \\	\Pair(M)\arrow[r,shift right] \arrow[r,shift left]& M	\end{tikzcd}}
\] 
with base map $(\tz,\sz)\colon G\to \Pair(M)$.
Explicitly, 
\begin{equation}\label{eq:explicit}
 x\sim_R(\zeta',\zeta)\ \Leftrightarrow\  \tz_{\TG}(x)=\alpha(\zeta'),\ \ \ 
\beta([l_{\wh{g}^{-1}}(x)])=\zeta-\Ad_{\wh{g}^{-1}}\zeta';\end{equation}
here $g$ is the base point of $x$, and $\wh{g}$ is any lift. The morphism $R$ defines a morphism of multiplicative Manin pairs 
\[ R\colon (\TG,TG)\da (\Pair(\EE),\Pair(A)).\]
\end{theorem}

Before starting the proof, let us address the apparent lack of symmetry in \eqref{eq:explicit}.
\begin{lemma} \label{lem:alt}
In the setting of Theorem \ref{th:main}, we have 
\[ \beta([l_{\wh{g}^{-1}}(x)])=\zeta-\Ad_{\wh{g}^{-1}}\zeta'
\ \ \Leftrightarrow \ \ \beta([r_{\wh{g}^{-1}}(x)])=\Ad_{\wh{g}}\zeta-\zeta'.\]
Furthermore, if this condition holds, 
\[ \tz_{\TG}(x)=\alpha(\zeta') \Leftrightarrow \sz_{\TG}(x)=\alpha(\zeta).\]
\end{lemma}
\begin{proof} 
The first equivalence follows from the $\wh{G}$-equivariance of $\beta$. Suppose this condition holds, and that 
$\tz_{\TG}(x)=\alpha(\zeta') $. To compute $\sz_{\TG}(x)=\sz_{\TG}(l_{\wh{g}^{-1}}(x))$, 
write $l_{\wh{g}^{-1}}(x)$ in terms of the decomposition 
$\TG|_M=(\TG)^{(0)}\oplus \ker(\tz_{\TG})|_M$, 
	\[ l_{\wh{g}^{-1}}(x)=\Ad_{\wh{g}^{-1}}(\tz_{\TG}(x))
	+(l_{\wh{g}^{-1}}(x)-\Ad_{\wh{g}^{-1}}(\tz_{\TG}(x))).\]
On $(\TG)^{(0)}$, the map $\sz_{\TG}$ acts as the identity, while on $\ker(\tz_{\TG})|_M\cong \core(\TG)$ 
it acts as  the core-anchor. 
It follows that 
\[ \sz_{\TG}(x)=\sz_{\TG}(l_{\wh{g}^{-1}}(x))=\Ad_{\wh{g}^{-1}}(\tz_{\TG}(x))+\varrho_{\TG}([l_{\wh{g}^{-1}}(x)]).\]
Using $\varrho_{\TG}=\alpha\circ \beta$ (see Lemma \ref{lem:coreanchor})
and the $\wh{G}$-equivariance of $\alpha$,
 this becomes 
\[  \sz_{\TG}(x)=\Ad_{\wh{g}^{-1}}\alpha(\zeta')
+\alpha(\zeta-\Ad_{\wh{g}^{-1}}(\zeta'))=\alpha(\zeta).
\]
The other direction is proved similarly. 
\end{proof}

\begin{proof}[Proof of Theorem \ref{th:main}]
The graph of a $\VB$-groupoid morphism is a $\VB$-groupoid in its own right, hence we may consider its units and core. 
We start out by \emph{defining} 
\[ R^{(0)}=\on{gr}(\alpha),\ \ \ \core(R)=\on{gr}(\beta).\]
This determines $R|_M=R^{(0)}\oplus \core(R)$. 

1) {\it $R|_M$ is a Lagrangian subbundle.} Observe that 
that 
\[
\on{rank}(R|_M)=
\on{rank}(\core(\TG))+\on{rank}(\core(\Pair(\EE)))
=\f{1}{2} \on{rank}(\Pair(\EE)\times \ol{\TG}).\]
To check that $R|_M$ is isotropic, recall that $(\TG)^{(0)},\ \Pair(\EE)^{(0)}$ are Lagrangian, hence 
$\on{gr}(\alpha)$ is isotropic. On the other hand, for $y=(v,\mu)\in \core(\TG)$ (with $v\in A,\ \mu\in T^*M)$)
we have 
\[ \l y,y\r=2\l\mu,\a(v)\r=\l \iota_A(v)+\a_\EE^*(\mu),\iota_A(v)+\a_\EE^*(\mu)\r=\l\beta(y),\beta(y)\r.\] 
Hence 
$\on{gr}(\beta)$ is isotropic. Finally, elements of $\on{gr}(\alpha)$ and $\on{gr}(\beta)$ are orthogonal, since 
$\beta$ is dual to $\alpha$. (See Remark \ref{rem:confusing}.)

2) {\it $R|_M$ extends to a $\wh{G}\times \wh{G}$-invariant Lagrangian subbundle $R\subset \Pair(\EE)\times \ol{\TG}$.} 
In Proposition \ref{prop:rui} of Appendix \ref{app:B}, we verify that the sum $\on{gr}(\alpha)\oplus \on{gr}(\beta)$ is invariant under the $K\times K$-action. Hence, $R=(\wh{G}\times \wh{G})\cdot R|_M$ is a well-defined subbundle. 
Since the $\wh{G}\times \wh{G}$-actions on both $\TG$ and $\Pair(\EE)$ are metric preserving, it follows that $R$ is a 
Lagrangian subbundle.

3) {\it $R$ is given by \eqref{eq:explicit}.} 
Elements $x\in \TG|_M$ decompose uniquely as $x=y+z$ where $y=\tz_{\TG}(x)\in (\TG)^{(0)}$ is a unit and 
$z=x-y=[x]\in \ker\tz_{\TG}|_M=\core(\TG)$.  Similarly, elements of $\Pair(\EE)|_M$ may be written in terms of 
the decomposition to unit and core, $(\zeta',\zeta)=(\zeta',\zeta')+(0,\zeta-\zeta')$. 
We have $x\sim_R (\zeta',\zeta)$ if and only if $y=\alpha(\zeta')$ and $\beta(z)=\zeta-\zeta'$. 
These are exactly the conditions \eqref{eq:explicit}, taking $\wh{g}$ to be trivial. To check 
 \eqref{eq:explicit} in general, it suffices to check that these conditions are invariant under the action of 
 the first factor of  
 $\wh{G}\times \wh{G}$. From the equivariance properties of the maps $\tz_{\TG}$ and $ \alpha$, it follows that, for $x\in \TG|_g$ and $\wh{h}\in \wh{G}$ with $\sz(h)=\tz(g)$,  
\[ \tz_{\TG}(l_{\wh{h}}(x))=\Ad_{\wh{h}}\tz_{\TG}(x)=
\Ad_{\wh{h}}(\alpha(\zeta'))=\alpha(\Ad_{\wh{h}}(\zeta')).\]
This shows that 
\[ x\sim_R(\zeta',\zeta)\ \Rightarrow\ 	l_{\wh{h}}(x)\sim_R(	\Ad_{\wh{h}}(\zeta'),\zeta),\]
as required. 

4) {\it $R$ is a $\VB$-subgroupoid for the horizontal structure.} 
Note 
\[ l_{\wh{g}_1^{-1}}r_{\wh{g}_2^{-1}}(x_1\circ x_2)=(l_{\wh{g}_1^{-1}} x_1)\circ (r_{\wh{g}_2^{-1}}x_2)\] 
for composable $x_1,x_2\in \TG$, and similarly for the pair groupoid. 
Hence, using the $\wh{G}\times \wh{G}$-equivariance, it suffices to show that $R|_M$ is a subgroupoid. 
Suppose $x_i\sim_R (\zeta_i',\zeta_i)$ for $i=1,2$, where $x_1,x_2\in \TG|_M$ are composable 
(i.e., have  the same base point), and $\zeta_1=\zeta_2'$. We want to show $x_1\circ x_2\sim_R (\zeta_1',\zeta_2)$. 
We have 
\[ \tz_{\TG}(x_1\circ x_2)=\tz_{\TG}(x_1)=\alpha(\zeta_1'),\]
which verifies the first condition in \eqref{eq:explicit}. On the other hand, $[x_1\circ x_2]=[x_1]+[x_2]$, and therefore 
\[ \beta([x_1\circ x_2])=\beta([x_1])+\beta([x_2])=(\zeta_1'-\zeta_1)+(\zeta_2'-\zeta_2)=\zeta_1'-\zeta_2,\]
verifying the second condition.

5) \emph{$R$ is an `almost' morphism of Manin pairs.} We next show 
\[ R(TG)\subset A\times A,\ \ TG\cap R^{-1}(0)=0.\] 
Suppose $v\in TG|_g\subset \TG|_g$ satisfies $v\sim_R (\zeta',\zeta)\in  \Pair(\EE)$. 
The condition $\tz_{TG}(v)=\alpha(\zeta')$ shows that $\pr_{A^*}(\zeta')=0$, thus 
$\zeta'\in A$. Similarly, $\sz_{TG}(v)=\alpha(\zeta)$ shows $\zeta\in A$. Suppose $v\sim_R (0,0)$. The second equation in \eqref{eq:explicit} shows that$\beta[(l_{\wh{g}^{-1}}(v)]=0$. 
Since the restriction of $\beta$ to $\core(TG)=A$ is just the inclusion $\iota_A$, this shows 
that $l_{\wh{g}^{-1}}(v)$is a unit:
\[ l_{\wh{g}^{-1}}(v)\in TM\subset TG|_M.\] 
But $\sz_{TG}(l_{\wh{g}^{-1}}(v))=\sz_{TG}(v)=\alpha(\zeta)=0$. We conclude $l_{\wh{g}^{-1}}(v)=0$, and hence $v=0$.

6) {\it Uniqueness.}
By Lemma \ref{lem:think} (which does not actually involve the bracket, only the metric), a Lagrangian $\VB$-subgroupoid of $\EE\oplus \ol{\EE}\oplus \ol{\TG}|_M\rra \Pair(\EE)^{(0)}\oplus \TG^{(0)}$, 
 use units are given by the graph of $\alpha$, is necessarily the sum of the graph of  
$\alpha$ and the graph of its 
dual map $\beta$ (see Remark \ref{rem:confusing}). By equivariance, this determines $R$ globally. 

7) {\it The image of $R$ under the anchor is tangent to the graph of $(\tz,\sz)$.}
	Suppose $x\sim_R(\zeta',\zeta)$. Write $x=v+\mu$ with $v\in TG,\ \mu\in T^*G$. Then 
	$\a_{\TG}(x)=v$. 
	From the equations $\tz_{\TG}(x)=\alpha(\zeta'),\ \sz_{\TG}(x)=\alpha(\zeta)$, we see in particular that $\tz_{TG}(v)=\a_\EE(\zeta'),\ \sz_{TG}(v)=\a_\EE(\zeta)$. 
This shows that the image of $R$ under the anchor is tangent to $\on{gr}(\tz,\sz)$. 

8) {\it $R$ is a  Dirac structure with support on $\on{gr}(\tz,\sz)$.} This involves some Courant bracket calculations, 
 which are a little technical and are therefore relegated to Appendix \ref{app:A}. 
\end{proof}

\subsection{Infinitesimally multiplicative version}
Following \cite[Section 6.2.1]{lib:th} we may formulate an `infinitesimally multiplicative version' of Theorem \ref{th:main}
for Manin pairs $(\EE,A)$, not requiring any integration of the Lie algebroid $A$. 

The pair $(\T A,TA)$ is an infinitesimally multiplicative Manin pair in the sense of \cite{lib:th}; in particular, $\T A$ is a Courant algebroid vertically 
(as a bundle over $A$) and a Lie algebroid horizontally (as a bundle over $(\T A)^{(0)}=T M\oplus A^*$).  Similarly, 
$(\on{flip}(T\EE),\on{flip}(TA))$ is an  infinitesimally multiplicative Manin pair; here $\on{flip}$ indicates that we interchanged the role of horizontal and vertical structures.  Thus, $\on{flip}(T\EE)$ is vertically a Courant algebroid over $\EE$, and horizontally a Lie algebroid over $TM\oplus A^*$. 

For any double vector bundle $D$, the fat bundles corresponding to the two side bundles act on $D$ via 
horizontal translations. (See, e.g., \cite{me:quot}.) We hence obtain translation actions of the vector bundle 
$\wh{A}=J^1(A)$ on $TA$ and on $T^*A$, hence also on $\T A$. In the case of $\on{flip}(T\EE)$, the horizontal fat bundle is the 
Atiyah algebroid $\on{At}(\EE)$. The Lie algebroid representation  of $\wh{A}=J^1(A)$ on $\EE$ is a 
morphism of Lie algebroids $J^1(A)\to \on{At}(\EE)$, resulting in a translation action of the vector bundle $\wh{A}$ 
on $\on{flip}(T\EE)$. 

The assertion is, then, that the map $\alpha\colon \EE\to TM\oplus A^*$ extends uniquely to a $\wh{A}$-equivariant 
$\CA$--morphism 
\[{\begin{tikzcd}		[column sep={8em,between origins},row sep={4.5em,between origins},]	\T A\arrow[r,Rightarrow]\arrow[d]& TM\oplus A^*\arrow[d] \\	A\arrow[r,Rightarrow]& M	\end{tikzcd}}
\stackrel{R_0}{\dasharrow}
{\begin{tikzcd}		[column sep={8em,between origins},row sep={4.5em,between origins},]	 \on{flip}(T\EE)\arrow[r,Rightarrow]\arrow[d]& \EE\arrow[d] \\	TM\arrow[r,Rightarrow]& M	\end{tikzcd}}
\] 
with base map $\a\colon A\to TM$. This morphism is infinitesimally multiplicative (i.e, $R_0\subset \on{flip}(T\EE)\times \T A$ is a sub-Lie algebroid horizontally), and defines a morphism of Manin pairs 
$R_0\colon (\T A,TA)\da (\on{flip}(T\EE),\on{flip}(TA))$. 

We omit the proof, which is similar to that of Theorem \ref{th:main}.

\section{Special cases and applications}\label{sec:special}
\subsection{The case $M=\pt$}\label{subsec:point}
If the manifold $M$ of units is a point $M=\pt$, then the groupoid $G\rra M$ is just a group, and a 
$\wh{G}$-equivariant Manin pair $(\EE,A)$ is a $G$-equivariant Manin pair 
\[ (\dd,\g)\] consisting of a metrized Lie algebra and Lagrangian Lie subalgebra. 
We may describe the 
morphism $ R\colon (\TG,TG)\da (\Pair(\dd),\Pair(\g))$ using left-trivialization  
\[ \TG\cong G\times (\g\oplus \g^*)\rra \g^*.\]
We have $\core(\TG)=\g$, with $\alpha=-\pr_{\g^*}\colon \dd\to \g^*$ being minus the projection and 
$\beta\colon \g\to\dd$ the inclusion. 
 Furthermore, 
$
\sz_{\TG}(g,\xi,\tau)=\tau,\ \ 
[r_{g^{-1}}(g,\xi,\tau)]=\Ad_g\xi$. We hence arrive at
\[(g,\xi,\tau)\sim_R (\zeta',\zeta)\ \ \Leftrightarrow\ \  \tau=-\pr_{\g^*}\zeta,\ \ \xi=\zeta-\Ad_{g^{-1}} \zeta'.\]
On the level of sections,  $R\subset \Pair(\dd)\times \ol{\TG}$ 
It is spanned by sections of the form 
\[ (0,\sigma,\sigma^L),\ \ (\Ad_g\zeta,\zeta,-\l \theta^L,\zeta\r)\]
with $\sigma\in \g$ and $\zeta\in \dd$. 


\begin{remark}
The morphism of Manin pairs $R_0\colon (\T\g,T\g)\da (\on{flip}(T\dd),\on{flip}(T\g))$ has the following description. 
Relative to the vertical structure, $\on{flip}(T\dd)$ is a Courant algebroid over $\pt$, and so is just a metrized Lie algebra.  One finds that as a Lie algebra, it is a semi-direct product
$\on{flip}(T\dd)=\dd\rtimes\dd$; we write its elements as 
pairs $(\zeta_v,\zeta_t)$ (where the subscripts are meant to suggest \emph{vertical} and \emph{tangent}). 
Then $R_0\subset \on{flip}(T\dd)\times \T\g$ is spanned by sections of the form
\[ ((\xi_v,0),\xi),\ \ (((\ad_a\zeta)_v,\zeta_t),\d\l\zeta,a\r).\]
for $\xi\in\g,\ \zeta\in\dd$, where $a$ is the parameter on the base manifold of $\T\g\to \g$. 
\end{remark}

\subsection{Lie group actions}
Let $(\dd,\g)$ be a Manin pair, and $M\times \dd\to TM$ a Lie algebra action with coisotropic stabilizers. As shown in \cite{lib:cou}, this data defines an \emph{action Courant algebroid} 
$\EE=M\times \dd$ 
over $M$, with anchor given by the action $M\times\dd\to TM$, and with Courant bracket  extending the Lie bracket on constant sections. The subbundle $A=M\times \g$ is a Dirac structure, and so defines a  Manin pair 
\[ (\EE,A)=(M\times\dd,M\times \g)\] 
over $M$. Suppose $G$ is a Lie group integrating $\g$, 
with an action on $M$ as well as a linear action on $\dd$, in such a way that 
\begin{itemize}
	\item[(i)] $(\dd,\g)$ is a $G$-equivariant Manin pair of Lie algebras,
	\item[(ii)] the $G$-action on $M$ integrates the $\g$-action, 
	\item[(iii)] the action map $M\times \dd\to TM$ is $G$-equivariant. 
\end{itemize}
The action groupoid 
\[ \G=G\times M\rra M\]
integrates $A$. (We are using the letter $\G$ since 
$G$ is already taken.) 

Observe that a $\G$-action on a manifold $Q$, along a  map $\Psi\colon Q\to M$, 
is the same as a $G$-action on $Q$ for which the map $\Psi$ is $G$-equivariant. In turn, a $\G$-action is equivalent to an action of $\wh{\G}=J^1(\G)$. In one direction, this follows because the action of $\wh{\G}$ is uniquely determined by the action of local bisections of $\G$. In the other direction, it follows because the map $\wh{\G}=J^1(\G)\to\G$ has a canonical splitting. (The splitting takes $(g,m)\in \G$ to the subspace $\{0\}\times T_mM\subset T_{(g,m)}G$. Both 
source and target restrict to isomorphisms on this subspace, which therefore is an element of $\wh{\G}$.)

The linear  $G$-action on $\EE$ by  
\[ g\cdot (m,\zeta)=(g\cdot m,\Ad_g(\zeta))\] preserves $A$. By the previous paragraph, it follows that  $(\EE,A)$ is a $\wh{\G}$-equivariant Manin pair. Hence, we obtain a morphism $R$ as in \eqref{eq:R}. To describe it explicitly, 
use left trivialization to identify $\TG=G\times (\g\oplus \g^*)$; thus 
\[ \T\G=G\times (\g\oplus \g^*)\times \T M\rra \g^*\times TM.\]
Denote its elements by $(g,\xi,\tau,v+\mu)$ with $g\in G,\xi\in\g,\tau\in\g^*,v\in T_mM,\mu\in T^*_mM$, and 
let  $\Phi\colon T^*M\to \g^*$ be the moment map for the cotangent lift, i.e. $\l\Phi(\mu),\xi\r=\l\mu,\xi_M(m)\r$.
\begin{proposition}
Using left trivialization as above, the morphism of Manin pairs
\[ R\colon (\T\G,T\G)\da (\Pair(M\times \dd),\Pair(M\times \g))\] 
is given by the condition that 
\[ (g,\xi,\tau,v+\mu)\sim_R (m',m,\zeta',\zeta)\]
if and only if $m'=\tz(g),\ m=\sz(g)$, and furthermore
\[\a(\zeta)=v,\ \ \pr_{\g^*}\zeta=\tau+\Phi(\mu),\ \ 
\xi+\a^*(\mu)=\zeta-\Ad_{g^{-1}}\zeta'.
\]
\end{proposition}
\begin{proof}
The $\VB$-groupoid $\T\G$ is the direct sum of the action groupoid $T\G=TG\times TM\rra TM$ and its dual
$T^*\G\rra \g^*$. In left trivialization, one obtains the formulas 
\[ \sz_{\T\G}(g,\xi,\tau,v+\mu)=(v,\ \tau+\Phi(\mu)),\ \ 
\tz_{\T\G}(g,\xi,\tau,v+\mu)=(g_*(v-\xi_M(m)),\Ad_g\tau)\]
with the multiplication of composable elements $(g_i,\xi_i,\tau_i,v_i+\mu_i),\ i=1,2$ given by 
\[ \Big(g_1\circ g_2,\ \Ad_{g_2^{-1}}\xi_1+\xi_2,\ \tau_2-\Ad_{g_2^{-1}}\Phi(\mu_1),\ v_2+\mu_1+\mu_2\Big).\]
The result is now a simple computation, using \eqref{eq:explicit}.
\end{proof}
  
\subsection{Quasi-symplectic groupoids}\label{subsec:quasisym}
We next explain how the integration result for twisted Dirac manifolds, due to Bursztyn-Crainic-Weinstein-Zhu  
\cite{bur:int}, may be obtained from Theorem \ref{th:main}.  These results correspond to Manin pairs $(\EE,A)$ 
 for which the Courant algebroid 
$\EE$ is \emph{exact} in the sense of \v{S}evera \cite{sev:let}: that is, $\on{rank}(\EE)=2\dim M$ and the anchor $\a_\EE\colon \EE\to TM$ is surjective. A morphism $R\colon \EE_1\da \EE_2$ between exact Courant algebroids, with base map 
$\Phi\colon M_1\to M_2$, is called \emph{exact} if 
the map 
$\a_{\EE_2}\times \a_{\EE_1}\colon R\to \on{gr}(T\Phi)\subset TM_2\times TM_1$ is surjective. 
We shall need the following fact:

\begin{proposition}\label{prop:exact}
	Let $(\EE,A)$ be a $\wh{G}$-equivariant Manin pair, such that the Courant algebroid  $\EE$ is exact. Then 
	the morphism $R\colon \TG \da \Pair(\EE)$  from Theorem \ref{th:main} is exact. 
\end{proposition}
\begin{proof}
We want to show that $\a(R)=T(\gr(\tz,\sz))$. By $\wh{G}\times \wh{G}$-equivariance, it suffices to check along $M\subset G\cong \on{gr}(\tz,\sz)$. 
But $R|_M$ is the direct sum of the graphs of 
\[\alpha=(\a_\EE,\pr_{A^*})\colon \EE\to TM\oplus A^*,\ \ \beta=(\iota_A,\a_\EE^*)\colon 	A\oplus T^*M\to \EE\]
(where we identify $\core(\TG)\cong \ker(\tz_{\TG})$, and similarly for $\Pair(\EE)$).
The image under $\a$ of these two summands is the direct sum of 
\[ \on{ran}(\a_\EE)=TM,\ \ A\cong \ker(\tz_{TG})|_M\] 
which indeed is  all of $TG|_M$. 
\end{proof}

According to \v{S}evera \cite{sev:let}, the choice of an isotropic splitting $j\colon TM\to \EE$ 
of an exact Courant algebroid identifies 
$\EE$ with the $\eta$-twisted standard Courant algebroid for a closed 3-form $\eta\in \Omega^3(M)$. That is, it determines an isomorphism $\EE\cong TM\oplus T^*M$, 
with the modified bracket obtained 
by adding a term $\iota_{X_1}\iota_{X_2}\eta$ to the right hand side of \eqref{eq:CA}. We shall denote this Courant algebroid by $\T_\eta M$. By an $\eta$-twisted Dirac structure on $M$, we mean a Dirac structure  $A\subset \T_\eta M$. 

Recall that a differential form $\alpha\in \Omega(G)$ on a Lie groupoid is \emph{multiplicative} if 
\[ \Mult_G^*\alpha=\pr_1^*\alpha+\pr_2^*\alpha,\] 
where $\Mult_G\colon G^{(2)}\to G$ is the groupoid multiplication and $\pr_1,\pr_2\colon G^{(2)}\to G$ are the two projections.
For exact $\CA$-groupoids $\EE\rra \EE^{(0)}$ over $G\rra G^{(0)}$, with \emph{multiplicative} isotropic splittings $TG\to \EE$, the 3-form $\eta\in \Omega^3(G)$ is multiplicative.  Exact $\CA$-morphisms 
\[ R\colon \T_{\eta}M\da \T_{\eta'}M',\] with base map  $\Phi\colon M\to M'$, are described by 2-forms $\omega\in \Omega^2(M)$, with $\d\omega=\eta-\Phi^*\eta'$. The relationship is given by
\[ v+\mu\sim_R v'+\mu' \Leftrightarrow v'=\Phi_*(v),\ \mu=\Phi^*\mu'+\iota_{v}\omega.\]
Exact $\CA$-groupoid morphisms are described by multiplicative 2-forms:

\begin{lemma}\label{lem:multiplicative}
Let $G,G'$ be Lie groupoids, with multiplicative closed 3-forms $\eta,\eta'$, and let 
$R\colon \T_\eta G\da \T_{\eta'}G'$ be an exact 
Courant morphism  given by a 2-form $\omega\in \Omega^2(G)$. 
Then $R$ is multiplicative if and only if $\omega$ is multiplicative. 
\end{lemma}	
\begin{proof}	
The multiplicativity of $\omega$ is equivalent to $\omega^\flat\colon TG\to T^*G$ defining a groupoid morphism. 
Hence, if $\omega$ is multiplicative  we see that if $v_i+\mu_i\sim _R v_i'+\mu_i'$ for $i=1,2$, and 
\[ v+\mu=(v_1+\mu_1)\circ  (v_2+\mu_2),\ \  v'+\mu'=(v_1'+\mu_1')\circ  (v_2'+\mu_2'),\] 
then   
$v+\mu\sim_R v'+\mu'$. This shows that $R$ is multiplicative. The same calculation also shows the converse.   
\end{proof}

Suppose   now that $(\EE,A)$ is an exact $\wh{G}$-equivariant Manin pair. Choose an isotropic 
splitting, identifying $\EE\cong T_\eta M$ for a closed 3-form $\eta$. By Proposition \ref{prop:exact}, 
the morphism $R\colon (\TG,TG)\da (\Pair(\EE),\Pair(A))$ from Theorem \ref{th:main} is exact, and so is given by a 
 2-form $\omega_G\in \Omega^2(G)$ satisfying 
 \begin{equation}\label{eq:closed}
 \d\omega_G=\sz^*\eta-\tz^*\eta.\end{equation}
Since $R$ is multiplicative, Lemma \ref{lem:multiplicative} shows that 
$\omega_G$ is multiplicative. 
The fact that $R$ is a morphism of Manin pairs translates to the fact that 
$A\subset \T_\eta M$ is the image of the map 
\[ (\a_A,\mathsf{b}_A)\colon A\to TM\oplus T^*M\] 
where 
\begin{equation}\label{eq:varrho} \mathsf{b}_A\colon A\to T^*M\end{equation} is the map $\mathsf{b}_A(\xi)(w)=\omega_G(v,w)$, for 
$v\in \ker(T\tz)|_M$ representing $\xi\in A$ and 
$w\in TM\subset TG|_M$. The uniqueness property in the definition of morphism of Manin pairs 
(Definition \ref{def:mor})
means that 
\begin{equation}\label{eq:nondegeneracy}
\ker(\omega_G)\cap \ker(T\tz)\cap \ker(T\sz)=0.\end{equation}
These are precisely the conditions defining the integration of the $\eta$-twisted Dirac structure $A\subset \T_\eta M$ 
to a quasi-symplectic groupoid, as in \cite{bur:cou}.  Note that if $\eta=0$ and $A\cap TM=0$, so that $A$ is the graph of a Poisson structure, then $\omega_G$ is \emph{symplectic}. We hence recover the fact that a Poisson manifold is integrable provided that its 
cotangent Lie algebroid is integrable.

\subsection{Integration of (quasi-) Lie bialgebroids}
Let $(\EE,A)$ be a Manin pair, and $B\subset \EE$ a Lagrangian subbundle complementary to $A$. 
The bundle metric on $\EE$ identifies $B\cong A^*$. By the theory of Liu-Weinstein-Xu \cite{liu:ma}, this choice of splitting defines on $A$ the structure of a 
\emph{quasi-Lie bialgeboid}. (It is a Lie bialgebroid if and only if $B$ is a Dirac structure.)  Conversely, given a 
quasi-Lie bialgebroid $A$, the direct sum $\EE=A\oplus A^*$ acquires the structure of a Courant algebroid, and $(\EE,A)$ is a 
Manin pair. We denote by 
\begin{equation}\label{eq:astar} \a_{A^*}\colon A^*\to TM\end{equation}
the restriction of $\a_\EE$ to $B\cong A^*$.

Suppose $G\rra M$ integrates $A$, and that $(\EE,A)$ is $\wh{G}$-equivariant. (There is no equivariance condition on $B$.) 
\begin{lemma}
The pre-image of 
\[ R^{-1}(\Pair(B))\]
of $\Pair(B)\subset \Pair(\EE)$ under the relation \eqref{eq:R}
 is the subgroupoid 
 \[ \on{gr}(\pi_G^\sharp)\rra \on{gr}(\a_{A^*})\]
 defined by a multiplicative bivector field $\pi_G$. If $B$ is a Dirac structure, then $\pi_G$ is Poisson. 
\end{lemma}
\begin{proof}
Since $\Pair(B)\subset \Pair(\EE)$ is a Lagrangian complement to $\Pair(A)$, 
its backward image under  $R$ is a 
Lagrangian complement to $TG$ (see e.g. \cite[Proposition 1.15]{al:pur}), and hence is the graph 
of a  bivector field $\pi_G$. (If $B$ is a Dirac structure, then so is $\Pair(B)$ and its backward image; hence $\pi_G$ is Poisson in that case.) Since $\Pair(B)$ is a subgroupoid and $R$ is multiplicative, the backward image $\on{gr}(\pi_G^\sharp)$ 
is a subgroupoid. The space of units of this subgroupoid is the graph of the restriction of 
$\pi_G^\sharp$ to $A^*\subset T^*G|_M$. This is the same as the map \eqref{eq:astar}. \end{proof}

In addition, by \cite{liu:ma} (see also \cite[Section 3.2]{lib:cou} or Example \ref{ex:burcou} below), the Lagrangian splitting $\EE=A\oplus B$ also defines a bivector field 
$\pi_M$ on $M$. Again, using that $R$ is a morphism of Manin pairs, it follows
(using, e.g, \cite[Section 3.5]{lib:cou}) 
that the base map $(\tz,\sz)$ is a (quasi-)Poisson map:
\[ \pi_G\sim_{(\tz,\sz)}(\pi_M,-\pi_M).\]
If $B$ is integrable (a Dirac structure), then $\pi_M$ is a Poisson structure. Furthermore, 
$\Pair(B)$ and its backward image are integrable, too, and so $\pi_G$ is a Poisson structure. 
Hence, in this case $\pi_G$ is a multiplicative Poisson structure on $G$, and  $(\tz,\sz)$ is a Poisson map.

In this way, the Mackenzie-Xu integration result for Lie bialgebroids \cite{mac:int} and its generalization to quasi-Lie bialgebroids 
\cite{igl:uni} may be seen as consequences of Theorem \ref{th:main}. In fact, we obtain a slight generalization since $G$ need not be source-simply connected.

\section{Hamiltonian spaces}\label{sec:hamiltonian}
\subsection{Hamiltonian $G$-spaces}
According to \cite{bur:cou,xu:mom},  manifolds with Manin pairs (i.e., generalized Dirac manifolds)
serve as moment map targets for moment map theories. The Hamiltonian spaces, in this context, are defined as morphisms of Manin pairs:
\begin{definition} \cite{bur:cou,igl:ham}
	A \emph{Hamiltonian $A$-space} for a  Manin pair $(\EE,A)$  is a manifold $P$ together with a
	morphism of Manin pairs 
	\begin{equation}\label{eq:lmap}
	L\colon (\T P,TP)\da  (\EE,A).\end{equation}
	The base map $\Phi\colon P\to M$ is called the \emph{moment map} of the Hamiltonian space.  
\end{definition}

\begin{example}\cite[Example 3.2]{bur:cou} \label{ex:burcou}
For every Manin pair $(\EE,A)$ over $M$, the manifold $P=M$ is a Hamiltonian $A$-space, with $\Phi=\on{id}_M$ the identity map, and with 
$L\subset \EE\times \ol{\T M}$ the sum of the graph of $\a_A\colon A\to TM$ and the graph of 
$\a_\EE^*\colon T^*M\to \EE$.  If $B$ is a Lagrangian complement to $A$, then its pre-image $L^{-1}(B)$ 
under this morphism 
is a Lagrangian complement to $TM$, and hence is the graph of a bivector field $\pi_M$.  
\end{example}

Given a Hamiltonian $A$-space, the Lie algebroid comorphism $TP\da A$ defines an action of the Lie algebroid $A\Ra M$ on $P$. 
In practice, one is interested in situations where this action integrates to a Lie groupoid action, with an equivariant moment map. 
Let $G\rra M$ be a Lie groupoid, with $\on{Lie}(G)=A$. An action of $G\rra M$ along $\Phi\colon P\to M$
lifts to actions of the jet groupoid $\wh{G}\rra M$ on the tangent and cotangent bundles of $P$; these combine into a $\wh{G}$-action on $\T P$ (along the base projection 
$\T P\to P$ followed by $\Phi$). 

\begin{definition}\label{def:G}
A \emph{Hamiltonian $G$-space} for a $\wh{G}$-equivariant Manin pair
$(\EE,A)$ is a $G$-manifold $P$ together with a $\wh{G}$-equivariant morphism of Manin pairs \eqref{eq:lmap}, 
in such a way that the $G$-action on $P$ differentiates to the $A$-action determined  by \eqref{eq:lmap}. 
\end{definition}

Below, we will spell out this notion of Hamiltonian $G$-space in terms of splittings. Bivector fields arise after choice of a Lagrangian complement to $A$; we shall show that this corresponds to (quasi-)Poisson groupoid actions on (quasi-)Poisson manifolds. 
On the other hand, when $\EE$ is exact, a choice of a isotropic splitting of the anchor map $\a_\EE\colon \EE\to TM$ produces a 
2-form theory. This will then relate Definition \ref{def:G} with  Xu's notion of Hamiltonian spaces for quasi-symplectic groupoids. 

We will need a result, showing that the equivariance of the morphism $L$ in \eqref{eq:lmap}
implies a $\TG$-module property with respect to  the multiplicative morphism of Manin pairs $R$ from 
Theorem \ref{th:main}. Recall first that a groupoid action $G\circlearrowright P$ naturally lifts to $\VB$-groupoid actions 
\[ TG \circlearrowright TP,\ \ \ 
T^*G \circlearrowright T^*P,\ \ \ 
\T G \circlearrowright \T P.\]
Here $TG\rra TM$ acts on $TP$ along $\Phi_{TP}=T\Phi\colon TP\to TM$ by tangent lift. 
Dually, $T^*G\rra A^*$ acts on $T^*P$ along the map $\Phi_{T^*P}\colon T^*P\to A^*$ given by 
\[ 
\l \Phi_{T^*P}(\mu),\,\xi\r=\l\mu,\xi_P|_p\r,\ \ \ \mu\in T^*P|_p,\ \xi\in A|_{\Phi(p)}
\] 
where $\xi_P|_p\in TP|_p$ is defined by $\xi_P|_p\sim_L \xi$
(the image of $\xi$ under the infinitesimal action map). 
The dual action is described in terms of the tangent action by 
\[ \nu'=\mu\cdot \nu\ \ 
\Leftrightarrow 
\l\nu',w'\r=\l\mu,v\r+\l \nu,w\r \ \mbox{ whenever } w'=v\cdot w,
\] 
for  $\nu,\nu'\in T^*P,\ \mu\in T^*G$. (Here $w',w\in TP,\ v\in TG$ are assumed to have the same base points as $\nu',\nu$ and $\mu$, respectively, and $\sz_{TG}(v)=\Phi_{TP}(w)$.) The tangent and cotangent action combine into an action of 
$\TG$ on $\T P$, along $\Phi_{\T P}=(\Phi_{TP},\Phi_{T^*P})$. 

The action $\TG\circlearrowright \T P$ is a $\CA$-groupoid action, in the sense that the action morphism $\T\A\colon \TG\times \T P\da \T P$ is a Courant relation \cite{lib:dir,vys:hit}. In particular, it has the property 
$\l x_1\cdot  y_1,\ x_2\cdot y_2\r=
\l x_1,x_2\r+\l y_1,y_2\r$.

\begin{proposition} \label{prop:ham}
	Let $P$ be a Hamiltonian $G$-space for the $\wh{G}$-equivariant Manin pair $(\EE,A)$, with moment map $\Phi\colon P\to M$ and morphism $L\colon \T P\da \EE$ (see \eqref{eq:lmap}). 
\begin{enumerate}
	\item For $y\in \T P$ and $\zeta\in \EE$, 	\[ y\sim_L\zeta \ \Rightarrow\ \Phi_{\T P}(y)=\alpha(\zeta).\]
	\item  The morphism $L$ is a module morphism with respect to the morphism $R\colon \TG\da \Pair(\EE)$ (see Theorem \ref{th:main}). That is, for  $x\in \TG,\ y\in \T P$ with $\sz_{\TG}(x)=\Phi_{\T P}(y)$, 
	\[ x\sim_R (\zeta',\zeta),\ \ y\sim_L \zeta\ \Rightarrow\ x\cdot y\sim_L(\zeta',\zeta)\cdot\zeta=\zeta'.\]
\end{enumerate}	 
\end{proposition}
The statement of (b) may be expressed as commutativity of the following  
diagram of Courant relations 
\[{\begin{tikzcd}
[column sep={12em,between origins},row sep={4.5em,between origins},]
\TG\times \T P\arrow[r,dashed,"\T\A"] \arrow[d, dashed,"R\times L"]& \T P\arrow[d, dashed,"L"] \\
\Pair(\EE)\times \EE\arrow[r,dashed] & \EE
\end{tikzcd}}\] 
where the horizontal maps are $\CA$-groupoid actions.
\begin{proof}
	(a) Write $y=v+\mu\in \T P|_p$. 	Since the image of $L\subset \EE\times \ol{\T P}$ under the anchor is tangent to the graph of $\Phi$, we have $ \Phi_{TP}(v)=T\Phi(v)=\a_\EE(\zeta)$. 	On the other hand, for $\xi\in A_{\Phi(p)}$, since \eqref{eq:lmap} is a morphism of Manin pairs, 	we have 	\[ \l \Phi_{T^*P}(\mu),\,\xi\r=\l\mu,\xi_P|_p\r=\l \zeta,\xi\r=\l\pr_{A^*}(\zeta),\xi\r,\]	hence $\Phi_{T^*P}(\mu)=\pr_{A^*}(\zeta)$.

(b) Using the $\wh{G}$-equivariance, it suffices to consider the 
case that $x\in \TG|_M$.  Suppose $x\sim_R (\zeta',\zeta),\ \ y\sim_L \zeta$, and let 
$y'=x\cdot y$. We want to show $y'\sim_L \zeta'$, i.e., $(\zeta',y')\in L$. 
Since $L$ is Lagrangian, it suffices to show that 
$\l (\zeta',y'),(\zeta'',y'')\r=\l\zeta',\zeta''\r-\l y',y''\r$ is zero, for all 
$(\zeta'',y'')\in L$ (with the same base points). We calculate as follows:
\[  \l y',y''\r= \l x\cdot y,\ \Phi_{\T P}(y'')\cdot y''\r=\l x,\Phi_{\T P}(y'')\r+\l y,y''\r.\]
From $x\sim_R(\zeta',\zeta)$, $\Phi_{\T P}(y'')\sim_R(\zeta'',\zeta'')$, we obtain 
\[  \l x,\Phi_{\T P}(y'')\r=\l\zeta',\zeta''\r
-\l\zeta,\zeta''\r.\]
Similarly,  $y\sim_L\zeta,\ y''\sim_L \zeta''$ implies 
\[ \ \ \l y,y''\r=\l\zeta,\zeta''\r.\]
This shows  $\l y',y''\r=\l\zeta',\zeta''\r$ as desired. 
\end{proof}
\subsection{Bivector fields}

An action $\A\colon G\times_M P\to P$ of a Poisson groupoid $G\rra M$ on a Poisson manifold 
$P$, with moment map $\Phi\colon P\to M$, is called a \emph{Poisson action} if the graph of the action morphism, 
\[ \on{gr}(\A)=\{(g\cdot p,g,p)|\ \sz(g)=\Phi(p)\}\subset P\times (G\times P)^-\]
 is a coisotropic submanifold \cite{wei:coi}. This admits a reformulation in terms of the action of $\TG$  on  
$\T P$:
\begin{proposition}\label{prop:poissonaction}
An action of a Poisson groupoid $(G,\pi_G)$ on a Poisson manifold $(P,\pi_P)$ is a Poisson action if and only if 
the subgroupoid $\on{gr}(\pi_G^\sharp)\subset \TG$ preserves $\on{gr}(\pi_P^\sharp)\subset \T P$.	
\end{proposition} 
\begin{proof}
By definition of the action of $T^*G$ on $T^*P$, the elements of the conormal bundle of $\on{gr}(\A)$ 
are exactly the elements of $T^*(P\times (G\times P))$ of the form $(-\mu\cdot \nu,\mu,\nu)$. 	
The condition that $\on{gr}(\A)$ is coisotropic means that $(\pi_P,-\pi_G,-\pi_P)^\sharp$ takes this to the 
tangent bundle of $\on{gr}(\A)$. The resulting condition 
\[ \pi_P^\sharp(\mu\cdot \nu)=\pi_G^\sharp(\mu)\cdot \pi_P^\sharp(\nu)\]
says that the action of $\on{gr}(\pi_G^\sharp)\subset \TG$ preserves $\on{gr}(\pi_P^\sharp)$. 
\end{proof}

Let $P$ be a Hamiltonian $G$-space for the $\wh{G}$-equivariant Manin pair $(\EE,A)$. 
Given a Dirac structure $B$ complementary to $A\subset \EE$, its backward image under $L$ defines a Dirac structure  complementary to $TP\subset \T P$, which hence is the graph of a Poisson structure $\pi_P$. One also has a Poisson structure  $\pi_M$ defined by the decomposition $\EE=A\oplus B$, and the map $\Phi$ is a Poisson map: 
\[ \pi_P\sim_\Phi \pi_M.\]
\begin{proposition}
Let $P$ be a  Hamiltonian $G$-space for $(\EE,A)$, and choose a splitting $B\subset \EE$ as above, defining Poisson structures 
$\pi_G,\ \pi_P$. 
Then the moment map $\Phi\colon P\to M$ and the 
action morphism $\A\colon G\times P\da P$ are Poisson. 	
\end{proposition}
\begin{proof}
Let $x=\mu+\pi^\sharp(\mu)\in \on{gr}(\pi_G^\sharp)$ and $y=\nu+\pi^\sharp(\nu)\in \on{gr}(\pi_P^\sharp)$ with 
$\sz_{\TG}(x)=\Phi_{\T P}(y)$. 
By definition of the backward image, there are unique elements $(\zeta',\zeta)\in \Pair(B)$ and $\zeta''\in B$ such that 
$x\sim_R (\zeta',\zeta),\ y\sim_L \zeta''$. Since 
\[ \sz_{\TG}(x)=\alpha(\zeta),\ \ \Phi_{\T P}(y)=\alpha(\zeta'')\]
(see Proposition \ref{prop:ham}), we obtain $\alpha(\zeta)=\alpha(\zeta'')$. Since the map $\alpha\colon \EE\to TM\oplus A^*$ restricts to an injection 
on $B$ (the map $\pr_{A^*}\colon \EE\to A^*$ restricts to an isomorphism $B\cong A^*$), 
this holds if and only if $\zeta''=\zeta$. From Proposition \ref{prop:ham} (b), it then follows that $x\cdot y\sim_L \zeta'$. 
Since $\zeta'\in B$, this implies that $x\cdot y\in \on{gr}(\pi_P^\sharp)$. 
By Proposition \ref{prop:poissonaction}, this  proves that the $G\rra M$-action on $P$  is Poisson.  
\end{proof}	

\begin{remark}
If $B\subset \EE$ is any Lagrangian subbundle complementary to $A$, not necessarily integrable, then the bivector fields $\pi_G,\ \pi_P$ are only quasi-Poisson. The results above extend to the quasi-setting, with the obvious modifications. 
\end{remark}

\subsection{Hamiltonian spaces for quasi-symplectic groupoids}
We shall use the following fact:
\begin{proposition}\label{lem:anotherlemma}
	Let $G\rra M$ be a groupoid acting on a manifold $P$ along $\Phi\colon P\to M$, and let 
	$\omega_P\in \Omega^2(P),\ \omega_G\in \Omega^2(G)$ be 2-forms, where $\omega_G$ is multiplicative.
	Let $\A\colon G\,_\sz\!\times_\Phi\, P\to P$ be the action morphism, and $\pr_1,\pr_2$ the projections 		
	from $G\,_\sz\!\times_\Phi\, P$ to the two factors.
	Then  the following are equivalent:
	\begin{enumerate}
		\item $\A^*\omega_P=\pr_2^*\omega_P+\pr_1^*\omega_P$, 
		\item  the maps $\omega_P^\flat,\ \omega_G^\flat$ intertwine the $TG$-action on $TP$ with the $T^*G$-action on $T^*P$, 
		\item  the graph $\on{gr}(\omega_P^\flat)\subset \T P$ is invariant under the action of the subgroupoid $\on{gr}(\omega_G^\flat)\subset \T G$. 
	\end{enumerate}
\end{proposition}
This is a simple extension of the characterization of multiplicative 2-forms on groupoids, and is proved similarly. 

Suppose $(\EE,A)$ is a $\wh{G}$-equivariant Manin pair, where $\EE$ is exact. The choice of an isotropic splitting 
$j\colon TM\to \EE$  identifies $\EE\cong \T_\eta M$ for a closed 3-form $\eta$. 
By the discussion in Section \ref{subsec:quasisym}, the morphism $R$ from Theorem \ref{th:main} is automatically exact, and is described by a multiplicative 2-form $\omega_G\in \Omega^2(G)$, making $G\rra M$ into a quasi-symplectic groupoid.

Suppose $P$ is a Hamiltonian $G$-space for $(\EE,A)$. We assume that the morphism $L$ \eqref{eq:lmap} is exact, that is, 
its image under the anchor surjects onto $\on{gr}(T\Phi)$. (This is not automatic; for instance, the morphism 
from Example \ref{ex:burcou} need not be exact.)   Then $L$ is represented by a 2-form $\omega_P$, with the property 
\[ \d\omega_P=-\Phi^*\eta.\]
By Proposition \ref{prop:ham}, $L$ is a module morphism with respect to $R$. 
The same argument as for Lemma \ref{lem:multiplicative} shows that this is 
equivalent to the fact that  the maps $\omega_P^\flat,\ \omega_G^\flat$ intertwine the $TG$-action on $TP$ with the $T^*G$-action on $T^*P$. By Proposition \ref{lem:anotherlemma}, this is equivalent to 
$\A^*\omega_P=\pr_2^*\omega_P+\pr_1^*\omega_P$. On the other hand, the uniqueness property in the definition of
morphism of Manin pairs is equivalent to the condition $\ker(\omega_P)\cap \ker(T\Phi)=0$. These two conditions are exactly the defining properties of a Hamiltonian space $(P,\omega_P)$ for a quasi-symplectic groupoid $(G,\omega_G)$ as in 
\cite{xu:mom}. 

\subsection{Morita equivalence} 
From the notion of Hamiltonian $G$-spaces for $\wh{G}$-equivariant Manin pairs, it is straightforward to define Morita equivalence of such Manin pairs. 

Recall that a Morita equivalence between Lie groupoids $G_1\rra M_1,\ G_2\rra M_2$ is a manifold $Q$ 
(the \emph{Hilsum-Skandalis bimodule})  with commuting 
left $G_1$-actions and right $G_2$-actions, 
\begin{equation}\label{eq:moritaequivalence}
\begin{tikzcd}
[column sep={5.5em,between origins},
row sep={4em,between origins},]
G_1    \arrow[d,shift left]\arrow[d,shift right]   & Q \arrow [dr, "J_2"'] \arrow[dl, "J_1"]  & G_2    \arrow[d,shift left]\arrow[d,shift right] \\
M_1 & & M_2
\end{tikzcd}
\end{equation}
where both actions are principal, and $J_1$ is the quotient map for the $G_2$-action while $J_2$ is the quotient map for the $G_1$-action. Note that a right action may be turned into a left action through 
groupoid inversion; hence we may think of $Q$ as a $G_1\times G_2$-manifold.  
Given another groupoid $G_3\rra M_3$  and a $G_2-G_3$-bimodule $Q'$, the 
manifold \[Q\circ Q'=(Q\times_{M_2} Q')/G_2\] serves as a $G_1-G_3$ bimodule for the underlying groupoids. For any Lie group $G$, the choice of $Q=Q$ with the $G\times G$-action by left and right multiplications defines a Morita equivalence of $G$ with itself; note that this serves as a unit under composition of Morita equivalences.

Suppose $(\EE_i,A_i),\ i=1,2$ are $\wh{G}_i$-equivariant Manin pairs, where $G_i\rra M_i$ are integrations of $A_i\Ra M_i$. A \emph{Morita bimodule} for these Manin pairs is a $G_1-G_2$ Morita bimodule $Q$ for the 
underlying groupoids, together with a $\wh{G}_1\times \wh{G}_2$-equivariant morphism of Manin pairs
\[ S\colon (\T Q,TQ)\da (\EE_1\times \ol{\EE}_2,A_1\times A_2)\]
having $(J_1,J_2)$ as its base map. The composition of such Morita bimodules for Manin pairs is given by reduction of Courant algebroids and their morphisms, similar to the discussions in \cite{bur:red} and 
\cite{cab:dir}. It is described by a commutative diagram of Courant morphisms
\[{\begin{tikzcd}
		[column sep={12em,between origins},row sep={4.5em,between origins},]
		\T Q\times \T Q'\arrow[r,dashed,"S\times S'"] \arrow[d, dashed]& \EE_1\times \ol{\EE}_2\times \EE_2\times \ol{\EE}_3\arrow[d, dashed] \\
		\T(Q\circ Q')\arrow[r,dashed,"{S\circ S'} "] & \EE_1\times \ol{\EE}_3
\end{tikzcd}}\] 
Here, the right vertical Courant morphism is defined by $(\zeta_1,\zeta_2,\zeta_2,\zeta_3)\sim
(\zeta_1,\zeta_3)$ while the left vertical morphism is given in terms of the inclusion 
$i\colon  Q\times_{M_2} Q'\to Q\times Q'$ and projection $\pi\colon Q\times_{M_2} Q'\to Q\circ Q'$ as 
$(\T \pi)\circ (\T i)^{-1}$ (where the superscript means the inverse relation, obtained by reversing arrows).
The fact that the bottom horizontal map is a Courant morphism follows from the theory in \cite{lib:dir}.

Given a $\wh{G}$-equivariant Manin pair $(\EE,A)$, the identity map of $(\EE,A)$ 
may be seen as a Morita equivalence, with the morphism $R$ from \eqref{eq:R} playing the role of $S$.

\begin{appendix}
\section{Courant bracket calculations}\label{app:A}

In this appendix, we will finish the proof of Theorem \ref{th:main}.
		To show that $R$ is a Dirac structure with support on $\on{gr}(\tz,\sz)$, it remains to show that for any two sections of $\Pair(\EE)\times \TG$ whose restriction to $\on{gr}(\tz,\sz)$ 
		takes values in $R$, their Courant bracket again takes values in $R$. Since the action of bisections of $G$ 
		by left multiplication preserves Courant brackets, and since $R$ is $\wh{G}$-invariant, it suffices to check this along 
		$\Pair(M)\times M\subset \Pair(M)\times G$. Furthermore, it suffices to check on any collection of sections spanning 
		$R$. Using $R|_M=R^{(0)}\oplus \core(R)$, we will use sections of the units and 
		core of $R$, suitably extended to a neighborhood of $\Pair(M)\times M$. 
		
		Given $\gamma \in \Gamma(\core(\TG))$, the section 
		\[ \psi(\gamma)=((0,\beta(\gamma),\gamma^L)\in \Gamma(\Pair(\EE)\times \T G)\] 
		restricts to a section of $\core(R)=\on{gr}(\beta)$. Sections of this type span all of $\core(R)$.
		To compute brackets between sections of this type, we separate 
		into the cases $\gamma=\xi\in \Gamma(A)$ (thus $\beta(\gamma)=\iota_A(\xi)$ 
		and $\gamma=\nu\in \Gamma(T^*M)$ (thus $\beta(\gamma)=\a_\EE^*\nu$). 
		
		\begin{lemma}
			We have 
			\[ \Cour{\psi(\xi_1),\psi(\xi_2)}=\psi([\xi_1,\xi_2]),\ \  
			\Cour{\psi(\xi),\psi(\nu)}=\psi(\L_{\a_A(\xi)}\nu),\ \ 
			\Cour{\psi(\nu_1),\psi(\nu_2)}=0
			\]
			for $\xi_1,\xi_2,\xi\in \Gamma(A)$ and $\nu_1,\nu_2,\nu\in \Gamma(T^*M)$. 
		\end{lemma}
		\begin{proof}
			The first identity follows since $A\subset E$ is a Dirac structure, and using the bracket relations for 
			left-invariant sections of $TG$. The second identity follows from 
			$\Cour{\iota_A(\xi),\a_\EE^*(\mu)}=\a_\EE^*(\L_{\a(\xi)}\nu)$
			and $\Cour{\xi^L,\nu^L}=(\L_{\a(\xi)}\nu)^L$. The third identity holds 
			since $\Cour{\a_\EE^*(\nu_1),\a_\EE^*(\nu_2)}=0$ and $\Cour{\nu_1^L,\nu_2^L}=0$. 
		\end{proof}
		
		We also need sections of $\Pair(\EE)\times \TG$ whose restrictions to $\gr(\tz,\sz)$ span 
		$R^{(0)}=\on{gr}(\alpha)$.
		A section $\zeta\in \Gamma(\EE)$, determines a section
		\[ \phi_0(\zeta)=\big((\zeta,\zeta),\alpha(\zeta)\big)\in \Gamma(\Pair(\EE)\times \T G|_M)\]
		over $\Pair(M)\times M$. There is no \emph{canonical} extension of $\phi_0(\zeta)$  
		to a section over $\Pair(M)\times G$. However, we may construct an extension locally, near $\Pair(M)\times \{m_0\}$ for any given $m_0\in M$. The extension depends on the choice of a frame $\xi_1,\ldots,\xi_k\in \Gamma(A|_U)$ over an 
		open neighborhood $U\subset M$ of $m_0$. 
		
		\begin{lemma}
			The sections $\phi_0(\zeta),\ \zeta\in \Gamma(\EE)$ admit extensions to sections  
			$\phi(\zeta)$ of $\Pair(\EE)\times \T G$ over some neighborhood of $\Pair(M)\times U$, such that 
			$\phi(\zeta)$ restrict to sections of $R$, and satisfy 
			\[ \Cour{\psi(\xi_i),\phi(\zeta)}|_{\Pair(M)\times U}=0.\]
		\end{lemma}
		\begin{proof}
			The vector fields 	
			\[ X_i=\a(\psi(\xi_i))=(0,\a_A(\xi_i),\xi_i^L)\]
			span the normal bundle to $\Pair(M)\times U$. 
			The vector fields $X_i$ lift to linear vector fields $\wt{X}_i$ on the total space of $\Pair(\EE)\times \TG$, in such a way that 
			the resulting Lie derivative on sections is given by 
			\[ \L_{\wt{X}_i}=\Cour{\psi(\xi_i),\cdot}.\] 
			Note that $\wt{X}_i$ are tangent to $R$. 
			
			The vector fields $X_i$ determines a germ of a tubular neighborhood embedding 
			\[ F\colon (\Pair(M)\times U)\times \R^k\to \Pair(M)\times G
			\]
			(defined locally, after shrinking $U$ if needed and replacing $\R^k$ by some small open neighborhood of $0$)
			taking $p\in   \Pair(M)\times U$ and 
			$y=(y_1,\ldots,y_k)\in R^k$
			to the time-1 flow of $\sum y_i X_i$ 
			applied to $p$. 
			
			Note that $X_i|_{\Pair(M)\times U}$ corresponds to $\f{\p}{\p y_i}|_{\Pair(M)\times U}$ under this map. 
			Using $\wt{X}_i$, we obtain a lift of $F$ to a tubular neighborhood embedding 
			\[ \wt{F}\colon \big((\Pair(\EE)\times \TG)|_{\Pair(M)\times U}\big)\times \R^k\to \Pair(\EE)\times \TG.\] 
			This gives the desired extension of $\phi_0(\zeta)$ to $\phi(\zeta)$, with the property that $\phi(\zeta)$ restricts to sections of 
			$R$, and 
			$\Cour{\psi(\xi_i),\phi(\zeta)}|_{\Pair(M)\times U}=\big(\L_{\wt{X}_i}\phi(\zeta)\big)|_{\Pair(M)\times U}=0$. 
		\end{proof}
		
		Consider $y_1,\ldots,y_k$ as normal coordinates for the tubular neighborhood embedding. We have 
		$\pr_{A^*}\zeta=\sum_i \l\zeta,\xi_i\r \ \d y^i|_{U\times \{0\}}$. Hence, the extension $\phi(\zeta)$ is of the  form
		\begin{equation}\label{eq:phizeta}
		\phi(\zeta)=\Big((\zeta,\zeta)+O(y),\ \a_\EE(\zeta)+\sum_i\l \xi_i,\zeta\r \d y_i +O(y)\Big)\end{equation}
		where $O(y)$ signifies sections that vanish for $y=0$. 
		
		By construction, the restriction of sections of the form $\psi(\gamma),\ \phi(\zeta)$ to 
		$\on{gr}(\tz,\sz)\subset \Pair(M)\times G$ takes values in $R$.  
		We have to show that this is also true for Courant brackets between such 
		sections. We only need to check at points of 
		\[  \gr_0(\tz,\sz)=\on{gr}(\tz,\sz)\cap (\Pair(M)\times U).\]
		(Observe that this is  just a copy of $U$, embedded diagonally.) The restriction of 
		$\Cour{\psi(\xi_i),\phi(\zeta)}$ is zero, by construction. For 
		$\Cour{\phi(\zeta),\phi(\zeta')}$, we have, using 
		\eqref{eq:phizeta}, 
		\[ \Cour{\phi(\zeta),\phi(\zeta')}=\phi(\Cour{\zeta,\zeta'})+
		\Big(O(y),\sum_i f_i \d y_i+O(y)\Big)\]
		for certain functions $f_i\in C^\infty(\Pair(M)\times U)$. To show that $f_i|_{\gr_0(\tz,\sz)}=0$, it suffices to show that the pairings with 
		$\psi(\xi_j)$ are zero.   Using the properties of Courant brackets, we have 
		\[ \l \Cour{\phi(\zeta),\phi(\zeta')},\psi(\xi_i)\r
		=\L_{\a(\phi(\zeta))} \l \phi(\zeta'),\psi(\xi_i)\r-\L_{\a(\phi(\zeta'))} \l \phi(\zeta),\psi(\xi_i)\r
		+\l \phi(\zeta'),\Cour{\psi(\xi_i),\phi(\zeta)}\r
		\]
		Upon restriction to $\gr_0(\tz,\sz)$, all of these terms vanish (note that the vector field $\a(\phi(\zeta))$ is tangent to 
		$\gr_0(\tz,\sz)$)). Hence 
		\[ \Cour{\phi(\zeta),\phi(\zeta')}|_{\gr_0(\tz,\sz)}
		=\phi(\Cour{\zeta,\zeta'})|_{\gr_0(\tz,\sz)}.
		\]
		takes values in $R$. Consider next the Courant bracket $\Cour{\phi(\zeta),\psi(\nu)}$.	
		Using that $\a(\psi(\nu))=0$ for $\nu\in \Gamma(T^*M)$ we find, after short calculations, 
		\begin{align*}
		\l 	\Cour{\phi(\zeta),\psi(\nu)},\phi(\zeta')\r&=\L_{\a(\phi(\zeta))}\l\psi(\nu),\phi(\zeta')\r
		+\l \psi(\nu),\Cour{\phi(\zeta),\phi(\zeta')}\r\\
		\Cour{\phi(\zeta),\psi(\nu)},\psi(\xi_i)\r
		&=\L_{\a(\phi(\zeta))}\l\psi(\nu),\phi(\xi_i)\r 
		-\l \psi(\nu),\Cour{\psi(\xi_i),\phi(\zeta)}\r\\
		\l 	\Cour{\phi(\zeta),\psi(\nu)},\psi(\nu')\r
		&=
		-\l\phi(\zeta),\Cour{\phi(\nu),\phi(\nu')}\r
		\end{align*}
		for all $\zeta'\in \Gamma(\EE),\nu'\in \Gamma(T^*M),\xi_i\in \Gamma(A)$. 
		All of these expressions vanish after restriction to $\gr_0(\tz,\sz)$. 
		Since the sections of the form $\phi(\zeta),\psi(\nu),\psi(\xi_i)$ span $R$ at 
		$\gr_0(\tz,\sz)$,  and $R$ is Lagrangian, the proof is complete. 
	
\section{More on $\VB$-groupoids}\label{app:B} 
\subsection{$\VB$-groupoids with trivial base groupoid}
Given vector bundles $B,C\to M$ and a bundle map 
\[ \varrho\colon C\to B\] 
the direct sum $J=C\oplus B$ has the structure of a $\VB$-groupoid, with trivial base groupoid, 
\begin{equation}\label{eq:trvibase} {\begin{tikzcd}
	[column sep={8em,between origins},row sep={4.5em,between origins},]
	C\oplus B\arrow[r,shift right] \arrow[r,shift left]\arrow[d]& B\arrow[d] \\
	M\arrow[r,shift right] \arrow[r,shift left]& M
	\end{tikzcd}}\end{equation}
The target and source maps are 
$ \tz_J(c,b)=b,\ \ \sz_J(c,b)=b+\varrho(c)$ and multiplication of composable elements reads as
\[ (c_1,b_1)\circ (c_2,b_2)=(c_1+c_2,b_1).\]
We have $\core(C\oplus B)=C$, with core-anchor the given map $\varrho$. 
Note that $C\oplus B\rra B$ is isomorphic to the action groupoid for the $C$-action on $B$ given by translation, 
$b\mapsto b-\varrho(c)$. 


Denote by 
\[ K\rra M\] the fat groupoid of the $\VB$-groupoid \eqref{eq:trvibase}. 
Since  $\sz_K=\tz_K$, hence $K$ is a family of Lie groups $K_m$. By definition, $K_m$ is the set of subspaces of $(C\oplus B)|_m$ that are complements to both $\ker(\tz_J)|_m$ and $\ker(\sz_J)|_m$. We may think of complements to $\ker(\tz_J)|_m\cong C_m$ as graphs of linear maps 
$\ff\colon B_m\to C_m$. Since $\sz_J(\ff(b)+b)=
(\on{id}_B+\varrho\circ \ff)(b)$, this leads to the description  
\[ K_m=\{\ff\in \Hom(B_m,C_m)|\ \det( \on{id}_B+\varrho\circ \ff )\neq 0\}\]
with the group(oid) multiplication given by 
\begin{equation}\label{eq:multinv} \ff_1\ff_2=\ff_1+\ff_2+\ff_2\circ \varrho\circ \ff_1.
\end{equation}
%
One verifies that $\on{Inv}_K(\ff)=-\ff\circ  (\on{id}_B+\varrho\circ \ff)^{-1}$. 
The Lie algebroid $\k\Ra M$ of $K\rra M$ is the vector bundle,  $\k=C\otimes B^*=\Hom(B,C)$. The anchor map is zero; thus $\k$ is a family of Lie algebras. Viewing the sections as bundle maps 
$\hh\colon B\to C$, the Lie algebroid bracket on $\k$ is given by 
\[ [\hh_1,\hh_2]=\hh_2\circ \varrho\circ \hh_1-\hh_1\circ \varrho\circ \hh_2.\]

As a special case of \eqref{eq:manyactions}, we have a groupoid action $ K\times K\circlearrowright C\oplus B$ by left 
and right multiplications, and the corresponding adjoint actions on $C$ and $B$. 

\begin{proposition}\label{prop:action}The representations of $K$ on $B,C$ are given by 
	\[
		\Ad^B_\ff=(\on{id}_B+\varrho\circ \ff)^{-1},\ \ \ 
	\Ad^C_\ff= (\on{id}_C+\ff\circ \varrho)^{-1}.
	\]
The $K$-action on $J=C\oplus B$ is the direct sum of these representations.
The left and right multiplications are given by 
\begin{align*}
	l_\ff(c+b)&=c+\Ad_\ff^C f(b)+\Ad_\ff ^B b,\\ 
	r_{\ff}^{-1}(c+b)&=\Ad_\ff^C(c-f(b))+b.
\end{align*}
\end{proposition}
\begin{proof}
	The adjoint action of $\ff\in K_m$ on $C_m\oplus B_m$ is computed as 
	\[ \Ad_\ff(c+b)=(\ff(b_1)+b_1)\circ (c+b)\circ (\on{Inv}_K(\ff)(b_2)+b_2)\]
	where $b_1,b_2\in B_m$ are uniquely determined by requiring that these elements are composable. 
	One obtains the conditions $b_1+\varrho(\ff(b_1))=b$
	and $b+\varrho(c)=b_2$. Using \eqref{eq:multinv}, one arrives at 
		\[ \Ad_\ff(c+b)=\Big(c-\ff\circ (\on{id}_B+\varrho\circ \ff)^{-1} \varrho(c)\Big)
		+(\on{id}_B+\varrho\circ \ff)^{-1} (b).\]
The first term may be written as $ (\on{id}_C+\ff\circ \varrho)^{-1}(c)$. The calculation for $r_{\ff}^{-1}$ is similar. 
\end{proof}

\begin{remarks}\label{rem:comments}
	\begin{enumerate}
		\item\label{it:inf} By differentiation, the infinitesimal representations of $\k$ on $B,C$ are given by 
		\[
		\ad^B_h=-\varrho\circ \hh,\ \ \ad^C_h=-\hh\circ \varrho.\]
		\item The source map $\sz_J$ intertwines $l_\ff$ with the trivial action on 
		$B$, while $\tz_J$ intertwines it with $\Ad^B_\ff$; for  $r_{\ff}^{-1}$ the roles of source and target 
		are reversed. So, 
		\[ \sz_J\circ l_\ff=\sz_J,\ \ 
		\tz_J\circ l_\ff=\Ad^B_\ff\circ \tz_J,\ \ 
		\sz_J\circ r_\ff^{-1}=\Ad^B_\ff\circ \sz_J,\ \ 
		\tz_J\circ r_\ff^{-1}=\tz_J.
		\]
		\item 	If $\varrho=0$, we simply have $K=\Hom(B,C)$. At the  opposite extreme, if 
		$C=B,\ \varrho=\on{id}$, 
		the map $\Ad^B$ gives an identification 
		$K=\on{GL}(B)$.  
		\item In general, $K_m$ retracts onto $\on{GL}(C_m/\ker(\varrho_m))$. In particular, it has exactly two connected components unless $\varrho_m=0$. 
	\end{enumerate}
\end{remarks}

Replacing $\varrho$ with its dual map $\varrho^*\colon C^*\to B^*$, we obtain a $\VB$-groupoid $B^*\oplus C^*\rra C^*$. 
\begin{proposition}\label{prop:dualact}
The $\VB$-groupoid $B^*\oplus C^*\rra C^*$ defined by the dual map $\varrho^*\colon C^*\to B^*$ is the dual to 
$C\oplus B\rra B$, relative to the pairing 
\[ \l \gamma+\beta,c+b\r=\l \gamma,c\r+\l\beta,b\r+\l\beta,\varrho (c)\r.\]
The corresponding adjoint actions of $K$ on $B^*,C^*$ are 
\[ \Ad^{B^*}_\ff(\beta)=\on{id}_{B^*}+f^*\varrho^*,\ \ \ \Ad_\ff^{C^*}(\gamma)=\on{id}_{C^*}+\varrho^* f^*,\]
and the  actions by left multiplication and right multiplication  on $C^*\oplus B^*$ are given by 
\[ l_\ff(\beta+\gamma)=\beta-f^*\gamma+\Ad_\ff^{C^*}\gamma,\]
\[ r_{\ff^{-1}}(\beta+\gamma)=\Ad_\ff^{B^*}\beta+f^*\gamma+\gamma.\]
\end{proposition}

\begin{proof}
These formulas are obtained by direct calculation. In particular, one verifies that the proposed pairing satisfies 
$\l \tau_1\circ \tau_2,z_1\circ z_2\r=\l\tau_1,z_1\r+\l \tau_2,z_2\r$ for composable elements 
$\tau_i\in B^*\oplus C^*,\ z_i=C\oplus B$. Similarly, one checks $\l l_\ff \tau,\ l_\ff z\r=\l \ff,z\r
=\l r_\ff^{-1}\tau,r_{\ff}^{-1}z\r$ for 
$\tau\in B^*\oplus C^*,\ z\in C\oplus B$. 	
\end{proof}

Given a $\VB$-groupoid $J\rra B$ over $G\rra M$, as in \eqref{eq:vbg}, the restriction $J|_M\rra B$ is of the form
$J|_M=C\oplus B$ considered above, with $C=\core(J)\cong \ker(\tz_J)|_M$,  and with $\varrho$ the core-anchor. The fat groupoid $\wh{G}$ of $J$ fits into an exact sequence
\begin{equation}\label{eq:exact1}
1\to K\to \wh{G}\to G\to 1.
\end{equation}

\begin{remark}	
	Observe that the source fibers of $\wh{G}$ are principal $K_m$-bundles
	\[ \sz_{\wh{G}}^{-1}(m)\to \sz_G^{-1}(m)\]
	over the source fibers  of $G$. In particular, we see that even when the groupoid $G$ has source connected 
	(or	source-simply connected) fibers,  this is not usually the case 
	for $\wh{G}$. 
\end{remark} 
%


\subsection{The action of $K\rra M$ on  Lie algebroids and Manin pairs} 
Let $A\Ra M$ be a Lie algebroid. The construction of the previous section, taking $\varrho$ to be the anchor map 
$\a\colon A\to TM$, gives a groupoid $K\rra M$ with an action on $A$. The Lie algebroid 
$\k=\Hom(TM,A)=T^*M\otimes A$ fits into the jet sequence 
\[ 0\to \k\to J^1(A)\to A\to 0.\]
\begin{lemma}
The infinitesimal $\k$-action $h\mapsto \ad^A_h$ agrees with the action as a subalgebroid of $J^1(A)$. 
Given an integration $G\rra M$, the $K$-action $f\mapsto \Ad^A_\ff$ agrees with the action as a subgroupoid 
of $K\subset J^1(G)$. 	
\end{lemma}
\begin{proof}
	For holonomic sections $j^1(\sigma)$, the $J^1(A)$-representation is given by 
	\[  j^1(\sigma)\cdot\tau=[\sigma,\tau],\ \ 
	j^1(\sigma)\cdot X=[\a(\sigma),X]
	\]
	for $\tau\in \Gamma(A),\ X\in \Gamma(TM)$. 	The inclusion 
	$\k=T^*M\otimes A\to J^1(A)$ is given on the level of sections by $\d f\otimes \sigma\mapsto j^1(f\sigma)-f j^1(\sigma)$. 
	Once hence obtains 
	\begin{equation}\label{eq:br} (d f\otimes \sigma).\tau=[f\sigma,\tau]-f[\sigma,\tau]=-\l \d f,\a(\tau)\r \sigma\end{equation}
	consistent with $\ad_h^A=-h\circ \a$ (see Remark \ref{rem:comments}\ref{it:inf}). 
	If $A\Ra M$ integrates to a Lie groupoid $G\rra M$, then the $\wh{A}$-representation integrates to a $J^1(G)$-action on $A$. 
	The latter comes from the action of the fat groupoid $\wh{G}=J^1(G)$ on the $\VB$-groupoid $TG\rra TM$, and this restricts to the action of $K\subset \wh{G}$ on $TG|_M$. 
\end{proof}

Suppose now that $(\EE,A)$ is a Manin pair. Then the $K$-representation on $A$, given by 
$f\mapsto \Ad_f^A=(\id+\ff\circ \a)^{-1}$,  extends to a representation on $\EE$, by 
\[ f\mapsto U_\ff=(\on{id}_\EE+\iota_A\circ f\circ \a_\EE)^{-1}.\]
We also have the dual representation of $K$ on $\EE^*\cong \EE$, 
\[ f\mapsto  V_\ff=((U_\ff)^{-1})^*=\on{id}_\EE+\a_\EE^*\circ \ff^*\circ \pr_{A^*}.\]
The two representations commute, due to $\pr_{A^*}\!\circ\, \iota_A=0,\ \a_\EE\circ \a_\EE^*=0$, hence they define a
$K\times K$-representation. The diagonal action gives a self-dual representation 
\[ \Ad^\EE_\ff= U_\ff\circ  V_\ff.\]
Note that $V_\ff$ acts trivially on $A$; hence $\Ad^\EE_\ff$ restricts to $\Ad^A_\ff$. 
Writing $U_\ff=(U_{\on{Inv}_K(\ff)})^{-1}=\id_\EE-\iota_A\circ f\circ \Ad_\ff\circ \a_\EE$, we also have the following expression 
\begin{equation}\label{eq:2ndformula}
\Ad_\ff^\EE=\id_\EE+\a_\EE^*\circ f^*\circ \pr_{A^*}-\iota_A\circ f\circ \Ad_\ff^{TM}\circ\, \a_\EE.\end{equation}
\begin{remark}
Note that this may be written in terms of the maps $\alpha=(\a_\EE,-\pr_{A^*})$ and $\beta=(\iota_A,-\pr_\EE^*)$ from \eqref{eq:alphabeta}, as follows. 
We have the maps $\on{Inv}_K(f)=-f\circ \Ad_f^{TM}\colon TM\to A$ and $f^*\colon A^*\to T^*M$. 
Together, this give a map $(\on{Inv}_K(f),f^*)\colon TM\oplus A^*\to A\oplus T^*M$, which we may regard as a map from 
$(TG)^{(0)}\to \core(TG)$. We hence obtain 
\[ \Ad_\ff^\EE=\id_\EE+\beta\circ (\on{Inv}_K(f),f^*)\circ \alpha.\]
\end{remark}
Let 
\[ h\mapsto \ad^\EE_h=u(h)+v(h)\] 
be the infinitesimal action of $\k$, where $u(h)=\iota_A\circ h\circ \a_\EE$ and $v(h)=-u(h)^*
=-\a_\EE^*\circ h^*\circ \pr_{A^*}$. 

\begin{proposition}
The $\k$-action $\ad^\EE$ coincides with the action as a subalgebroid $\k\subset J^1(A)$. 
\end{proposition}
\begin{proof}
	The representation of $J^1(A)$ on $\EE$ is given by the Courant bracket, 
	$  j^1(\sigma)\cdot\tau=\Cour{\sigma,\tau}$. The action of non-holonomic sections $\d f\otimes \sigma$ is computed as 
	$ (d f\otimes \sigma).\tau=\Cour{f\sigma,\tau}-f\Cour{\sigma,\tau}$. After short calculation, one obtains
	\[
	(d f\otimes \sigma).\tau=-\l \d f,\a(\tau)\r \sigma+\l \sigma,\tau\r\,\a^*(\d f).\]
	In terms of $h=d f\otimes \sigma$, the right hand side is $u(\hh)\tau-u(\hh)^*\tau$. 
\end{proof}

\begin{lemma}
The $K$-action on $(\EE,A)$ satisfies 
\[ \a_\EE\circ \Ad^\EE_\ff=\Ad^{TM}_\ff\circ \a_\EE,\ \ \ \ \ \pr_{A^*}\circ\, \Ad^\EE_\ff=
\Ad^{A^*}_\ff\circ \pr_{A^*}.\]
\end{lemma}
\begin{proof}
For the first identity, we calculate
\[ \a_\EE\circ \Ad_\ff^\EE=\a_\EE\circ U_\ff=\a_\EE\circ (\id_\EE+\iota_A\circ f\circ \a_\EE)^{-1}=
(\id_{TM}+\a_A\circ f)^{-1}\circ \a_\EE=\Ad_\ff^{TM}\circ \a_\EE.\]	
The second identity is verified similarly. 
\end{proof}

Given an integration $G\rra M$ of $A\Ra M$, consider the $K\times K$-actions on $\TG|_M\rra TM\oplus A^*$ given by
 $(\ff_1,\ff_2)\mapsto l_{\ff_1}r_{\ff_2}^{-1}$ and 
on $\Pair(\EE)|_M=\EE\oplus \ol{\EE}\rra \EE$ by $(\ff_1,\ff_2)\mapsto (\Ad_{\ff_1}^\EE,\Ad_{\ff_2}^\EE)$. 
Recall the maps $\alpha\colon \Pair(\EE)^{(0)}\to (\TG)^{(0)}$ and $\beta\colon \core(\TG)\to \core(\Pair(\EE))$
from \eqref{eq:alphabeta}.

\begin{proposition}\label{prop:rui}
The subbundle $R|_M=\on{gr}(\alpha)+\on{gr}(\beta)\subset \EE\oplus \ol{\EE}\oplus \ol{\TG}|_M$
is $K\times K$-invariant. 
\end{proposition}
\begin{proof}
We show invariance under the action of $K\times \{1\}$; the argument for the second $K$-factor is similar. Note that 
$K\times \{1\}$ acts trivially on $\core(\TG)$ and $\core(\Pair(\EE))$
hence $\on{gr}(\beta)$ is $K\times \{1\}$-invariant. 
Consider on the other hand an element 
\[ (\zeta,\zeta,\alpha(\zeta))\in \on{gr}(\alpha)\]
defined by $\zeta\in \EE\cong  (\EE\times \EE)^{(0)}$.  Here 
\[ \alpha(\zeta)=\a_\EE(\zeta)-\pr_{A^*}\zeta\in TM\oplus A^*=(\TG)^{(0)}.\] 
The formulas from Propositions \ref{prop:action} and \ref{prop:dualact} describe the action $l_\ff$ on 
$TM$ and $A^*$. We obtain 
\begin{align*}
l_\ff \alpha(\zeta)&=l_\ff(\a_\EE(\zeta))-l_\ff(\pr_{A^*}\zeta)\\
&=\Ad_\ff^{TM}(\a_\EE(\zeta))- \Ad_\ff^{A^*}\pr_{A^*}(\zeta)
+\Ad_\ff^A f(\a_\EE(\zeta))+f^*\pr_{A^*}(\zeta)\\
&=\Ad_\ff^{\TG}\alpha(\zeta)+y\\
&=\alpha(\Ad_\ff^\EE\zeta)+y.
\end{align*}
where $y=\Ad_\ff^A f(\a_\EE(\zeta))+f^*\pr_{A^*}(\zeta)=-(\on{Inv}_K(f),f^*)(\alpha(\zeta))$ lies in $\core(\TG)$. Equation \eqref{eq:2ndformula} shows that 
\[ \beta(y)=\zeta-\Ad_\ff^\EE(\zeta).\]
Comparing with 
\[ (\ff,1)\cdot (\zeta,\zeta)=(\Ad_\ff^\EE\zeta,\zeta)=(\Ad_\ff^\EE\zeta,\Ad_\ff^\EE\zeta)+(0,\zeta-\Ad_\ff^\EE\zeta)\]
this confirms that $l_\ff(\alpha(\zeta))\sim_{R|_M} (f,1)\cdot (\zeta,\zeta)$.
\end{proof}

\begin{remark}
Since the 	$K\times K$-actions on $TG|_M=A\oplus TM,\ T^*G|_M=T^*M\oplus A^*$ 
are fully described in terms of $\a_A$, Proposition \ref{prop:rui} does not actually require the integration $G\rra M$.
\end{remark}

\begin{proposition}\label{prop:rr}
Let $(\EE,A)$ be a Manin pair, and $G\rra M$ a source-simply connected groupoid integrating $A\Ra M$. Then there is a unique 
action of $\wh{G}=J^1(G)$ on $\EE$, differentiating to the given action of $\wh{A}=J^1(A)$, and restricting 
to the action of $K\rra M$ described above.  	
\end{proposition}
\begin{proof}	
Let $J^1(G)_0$ denote the source-identity component of the groupoid  $J^1(G)$, and $\wt{J}^1(G)$ its 
source-simply connected cover (see \cite[Section 4.1]{cra:mul}). Then $ \wt{J}^1(G)$ is the source-simply connected integration of 
the jet algebroid $J^1(A)$. The representation of  $J^1(A)$ on $\EE$ integrates to a representation of $\wt{J}^1(G)$ on 
$\EE$.  Since $G$ is source-simply connected, the inclusion $K_0\hra J^1(G)_0$ gives an isomorphism of fundamental 
groups for the source fibers. The kernel of the covering projection $\wt{J}^1(G)\to J^1(G)_0$ coincides with that of 
$\wt{K}\to K_0$. But since the $\k$-action exponentiates to a $K_0$-action, it follows that this kernel acts trivially. 
We conclude that the $J^1(A)$-action integrates to an action of the groupoid $J^1(G)_0$. On the other hand, we also have the action of $K$. The two actions agree over $K_0 =J^1(G)_0\cap K$, and are compatible in the sense that 
\[ \wh{g}\cdot k\cdot \wh{g}^{-1}\cdot \zeta=(\Ad_{\wh{g}}k)\cdot \zeta,\]
for $\wh{g}\in J^1(G)_0,\ k\in K,\ \zeta\in \EE$, as is seen from the explicit formula for the $K$-action. It follows that the 
$J^1(G)_0$-action and $K$-action combine into an action of $J^1(G)$. 
\end{proof}
\end{appendix}


\begin{thebibliography}{10}
	
	\bibitem{al:pur}
	A.~Alekseev, H.~Bursztyn, and E.~Meinrenken, \emph{Pure spinors on {L}ie
		groups}, Ast\'{e}risque \textbf{327} (2009), 131--199.
	
	\bibitem{al:der}
	A.~Alekseev and P.~Xu, \emph{Derived brackets and {C}ourant algebroids},
	Unfinished manuscript (2002).
	
	\bibitem{alv:poi}
	D.~\'{A}lvarez, \emph{Poisson groupoids and moduli spaces of flat bundles over
		surfaces}, Adv. Math. \textbf{440} (2024), Paper no. 109523.
	
	\bibitem{bur:red}
	H.~Bursztyn, G.~Cavalcanti, and M.~Gualtieri, \emph{Reduction of {C}ourant
		algebroids and generalized complex structures}, Adv.~Math. \textbf{211}
	(2007), no.~2, 726--765.
	
	\bibitem{bur:int}
	H.~Bursztyn, M.~Crainic, A.~Weinstein, and C.~Zhu, \emph{Integration of twisted
		{D}irac brackets}, Duke Math.~J. \textbf{123} (2004), no.~3, 549--607.
	
	\bibitem{bur:cou}
	H.~Bursztyn, D.~Iglesias Ponte, and P.~Severa, \emph{Courant morphisms and
		moment maps}, Math. Res. Lett. \textbf{16} (2009), no.~2, 215--232.
	
	\bibitem{cab:dir}
	A.~Cabrera, M.~Gualtieri, and E.~Meinrenken, \emph{{Dirac geometry of the
			holonomy fibration}}, Comm. Math. Phys. \textbf{355} (2017), no.~3, 865--904.
	
	\bibitem{cat:int1}
	A.~Cattaneo, B.~Dherin, and A.~Weinstein, \emph{Integration of {L}ie algebroid
		comorphisms}, Port. Math. \textbf{70} (2013), no.~2, 113--144.
	
	\bibitem{cos:gro}
	A.~Coste, P.~Dazord, and A.~Weinstein, \emph{Groupo\"\i des symplectiques},
	Publications du {D}\'epartement de {M}ath\'ematiques. {N}ouvelle {S}\'erie.
	{A}, {V}ol.\ 2, Publ. D\'ep. Math. Nouvelle S\'er. A, vol.~87, Univ.
	Claude-Bernard, Lyon, 1987, pp.~i--ii, 1--62.
	
	\bibitem{cou:di}
	T.~Courant, \emph{Dirac manifolds}, Trans.~Amer.~Math.~Soc. \textbf{319}
	(1990), no.~2, 631--661.
	
	\bibitem{couwein:beyond}
	T.~Courant and A.~Weinstein, \emph{Beyond {P}oisson structures}, Action
	hamiltoniennes de groupes.~Troisi\`eme th\'eor\`eme de {L}ie (Lyon, 1986),
	Travaux en Cours, vol.~27, Hermann, Paris, 1988, pp.~39--49.
	
	\bibitem{cra:lec}
	M.~Crainic, R.~Fernandes, and I.~Marcut, \emph{Lectures on {P}oisson geometry},
	Graduate Studies in Mathematics, vol. 217, American Mathematical Society,
	Providence, RI, [2021] \copyright 2021.
	
	\bibitem{cra:mul}
	M.~Crainic, M.A. Salazar, and I.~Struchiner, \emph{Multiplicative forms and
		{S}pencer operators}, Math. Z. \textbf{279} (2015), no.~3-4, 939--979.
	
	\bibitem{gra:vb}
	A.~Gracia-Saz and R.~Mehta, \emph{{$\mathcal{VB}$}-groupoids and representation
		theory of {L}ie groupoids}, J. Symplectic Geom. \textbf{15} (2017), no.~3,
	741--783.
	
	\bibitem{igl:uni}
	D.~Iglesias-Ponte, C.~Laurent-Gengoux, and P.~Xu, \emph{Universal lifting
		theorem and quasi-{P}oisson groupoids}, J. Eur. Math. Soc. (JEMS) \textbf{14}
	(2012), no.~3, 681--731.
	
	\bibitem{igl:ham}
	D.~Iglesias-Ponte and P.~Xu, \emph{{Hamiltonian spaces for Manin pairs over
			manifolds}}, Preprint, 2008, https://arxiv.org/pdf/0809.4070.
	
	\bibitem{kar:ana}
	M.~V. Karas{\"e}v, \emph{Analogues of objects of the theory of {L}ie groups for
		nonlinear {P}oisson brackets}, Izv. Akad. Nauk SSSR Ser. Mat. \textbf{50}
	(1986), no.~3, 508--538, 638.
	
	\bibitem{kos:qua1}
	Y.~Kosmann-Schwarzbach, \emph{Quasi-big\`ebres de {L}ie et groupes de {L}ie
		quasi-{P}oisson}, C. R. Acad. Sci. Paris S\'er. I Math. \textbf{312} (1991),
	no.~5, 391--394.
	
	\bibitem{kos:jac}
	\bysame, \emph{Jacobian quasi-bialgebras and quasi-{P}oisson {L}ie groups},
	Mathematical aspects of classical field theory ({S}eattle, {WA}, 1991),
	Contemp. Math., vol. 132, Amer. Math. Soc., Providence, RI, 1992,
	pp.~459--489.
	
	\bibitem{lib:th}
	D.~{Li-Bland}, \emph{{LA-Courant algebroids and their applications}}, Ph.D.
	thesis, 2012.
	
	\bibitem{lib:cou}
	D.~Li-Bland and E.~Meinrenken, \emph{{C}ourant algebroids and {P}oisson
		geometry}, International Mathematics Research Notices \textbf{11} (2009),
	2106--2145.
	
	\bibitem{lib:dir}
	\bysame, \emph{{Dirac Lie groups}}, Asian Journal of Mathematics \textbf{18}
	(2014), no.~5, 779--816.
	
	\bibitem{lib:qua}
	D.~Li-Bland and P.~\v{S}evera, \emph{Quasi-{H}amiltonian groupoids and
		multiplicative {M}anin pairs}, International Mathematics Research Notices
	\textbf{2011} (2011), 2295--2350.
	
	\bibitem{liu:ma}
	Z.-J. Liu, A.~Weinstein, and P.~Xu, \emph{Manin triples for {L}ie
		bialgebroids}, J.~Differential Geom. \textbf{45} (1997), no.~3, 547--574.
	
	\bibitem{mac:gen}
	K.~Mackenzie, \emph{General theory of {L}ie groupoids and {L}ie algebroids},
	London Mathematical Society Lecture Note Series, vol. 213, Cambridge
	University Press, Cambridge, 2005.
	
	\bibitem{mac:lie}
	K.~Mackenzie and P.~Xu, \emph{Lie bialgebroids and {P}oisson groupoids}, Duke
	Math.~J. \textbf{73} (1994), no.~2, 415--452.
	
	\bibitem{mac:int}
	\bysame, \emph{Integration of {L}ie bialgebroids}, Topology \textbf{39} (2000),
	no.~3, 445--467.
	
	\bibitem{meh:qgr}
	R.~Mehta, \emph{{$Q$}-groupoids and their cohomology}, Pacific J. Math.
	\textbf{242} (2009), no.~2, 311--332.
	
	\bibitem{me:lec}
	E.~Meinrenken, \emph{Lectures on pure spinors and moment maps}, Poisson
	geometry in mathematics and physics, Contemp.~Math., vol. 450,
	Amer.~Math.~Soc., Providence, RI, 2008, pp.~199--222.
	
	\bibitem{me:poilec}
	E.~Meinrenken, \emph{{Poisson Geometry from a Dirac perspective}}, Letters in
	Mathematical Physics \textbf{108} (2018), no.~3, 447--498.
	
	\bibitem{me:quot}
	\bysame, \emph{Quotients of double vector bundles and multigraded bundles}, J.
	Geom. Mech. \textbf{14} (2022), no.~2, 307--329.
	
	\bibitem{ort:mu}
	C.~Ortiz, \emph{Multiplicative {D}irac structures on {L}ie groups}, C. R. Math.
	Acad. Sci. Paris \textbf{346} (2008), no.~23-24, 1279--1282.
	
	\bibitem{roy:qua}
	D.~Roytenberg, \emph{Quasi-{L}ie bialgebroids and twisted {P}oisson manifolds},
	Lett.Math.Phys. \textbf{61} (2002), 123--137.
	
	\bibitem{sev:let}
	P.~{\v{S}}evera, \emph{{Letters to Alan Weinstein, 1998-2000}}, available at
	arXiv:1707.00265.
	
	\bibitem{vys:hit}
	J.~Vysok\'y, \emph{Hitchhiker's guide to {C}ourant algebroid relations}, J.
	Geom. Phys. \textbf{151} (2020), 103635, 33.
	
	\bibitem{wei:sym}
	A.~Weinstein, \emph{Symplectic groupoids and {P}oisson manifolds}, Bull. Amer.
	Math. Soc. (N.S.) \textbf{16} (1987), no.~1, 101--104.
	
	\bibitem{wei:coi}
	\bysame, \emph{Coisotropic calculus and {P}oisson groupoids}, J. Math. Soc.
	Japan \textbf{40} (1988), no.~4, 705--727.
	
	\bibitem{xu:mor}
	P.~Xu, \emph{Morita equivalence of {P}oisson manifolds}, Comm. Math. Phys.
	\textbf{142} (1991), no.~3, 493--509.
	
	\bibitem{xu:mom}
	\bysame, \emph{Momentum maps and {M}orita equivalence}, J. Differential Geom.
	\textbf{67} (2004), no.~2, 289--333.
	
\end{thebibliography}

\def\cprime{$'$} \def\polhk#1{\setbox0=\hbox{#1}{\ooalign{\hidewidth
			\lower1.5ex\hbox{`}\hidewidth\crcr\unhbox0}}} \def\cprime{$'$}
\def\cprime{$'$} \def\cprime{$'$} \def\cprime{$'$} \def\cprime{$'$}
\def\polhk#1{\setbox0=\hbox{#1}{\ooalign{\hidewidth
			\lower1.5ex\hbox{`}\hidewidth\crcr\unhbox0}}} \def\cprime{$'$}
\def\cprime{$'$} \def\cprime{$'$} \def\cprime{$'$} \def\cprime{$'$}
\providecommand{\bysame}{\leavevmode\hbox to3em{\hrulefill}\thinspace}
\providecommand{\MR}{\relax\ifhmode\unskip\space\fi MR }
\providecommand{\MRhref}[2]{%
	\href{http://www.ams.org/mathscinet-getitem?mr=#1}{#2}
}
\providecommand{\href}[2]{#2}

\end{document}